\documentclass[12pt, oneside]{amsart}

\advance \topmargin by -\headheight
\advance \topmargin by -\headsep
\evensidemargin \oddsidemargin
\usepackage{color}
\usepackage{enumitem}
\usepackage{bbm}

\usepackage{graphicx}
\usepackage[mathscr]{euscript}
\usepackage{amsthm,amssymb,amsmath}
\usepackage{subfigure}
\allowdisplaybreaks[4]

\begin{document}
\title{Measure doubling of small sets in  $\mathrm{SO}(3,\mathbb{R})$}

\author{Yifan Jing}
\address{Mathematical Institute, University of Oxford, Oxford OX2 6GG, UK}
\email{yifan.jing@maths.ox.ac.uk}

\author{Chieu-Minh Tran}
\address{Department of Mathematics, National University of Singapore, Singapore}
\email{trancm@nus.edu.sg}

\author{Ruixiang Zhang}
\address{Department of Mathematics, University of California Berkeley, CA, USA}
\email{ruixiang@berkeley.edu}

\thanks{YJ~was supported by Ben Green’s Simons Investigator Grant, ID:376201.}
\thanks{RZ was supported by the NSF grant DMS-2207281, NSF CAREER Award DMS-2143989, and the Alfred P. Sloan Foundation}

\subjclass[2020]{Primary 20G20; Secondary 43A75, 22E30, 03C20, 11B30}

\date{} % Activate to display a given date or no date

\newtheorem{theorem}{Theorem}[section]
\newtheorem{lemma}[theorem]{Lemma}
\newtheorem{corollary}[theorem]{Corollary}
\newtheorem{fact}[theorem]{Fact}
\newtheorem{conjecture}[theorem]{Conjecture}

\newtheorem{proposition}[theorem]{Proposition}
\theoremstyle{definition}
\newtheorem{remark}[theorem]{Remark}
\newtheorem{definition}[theorem]{Definition}
\newtheorem*{thm:associativity}{Theorem \ref{thm:associativity}}
\newtheorem*{thm:associativity2}{Theorem \ref{thm:associativity2}}
\newtheorem*{thm:associativity3}{Theorem \ref{thm:associativity3}}
\def\tri{\,\triangle\,}

\def\d{\,\mathrm{d}}
\def\BM{\mathrm{BM}}
\def\RR{\mathbb{R}}
\def\ZZ{\mathbb{Z}}
\def\TT{\mathbb{T}}
\def\LL{\mathscr{L}}
\def\id{\mathrm{id}}
\def\tmu{\tilde{\mu}}
\def\pr{\mathrm{p}}
\newcommand\NN{\mathbb N}
\newcommand{\Case}[2]{\noindent {\bf Case #1:} \emph{#2}}
\newcommand\inner[2]{\langle #1, #2 \rangle}
\def\aL{\mathfrak{l}}
\def\aee{=_\ult}

\renewcommand{\epsilon}{\varepsilon}

\def\ult{\mathcal{U}}
\def \Sot{\mathrm{SO}(3,\RR)}
\def \Sod{\mathrm{SO}(d, \RR)}
\def \Gld{\mathrm{GL}(d, \RR)}

\def\cA{\mathscr{A}}
\def\cB{\mathcal{B}}
\def\cP{\mathcal{P}}
\def\cD{\Sigma}
\def\cAe{\mathscr{A}_{\epsilon}}
\def\cAet{\mathscr{A}_{\epsilon, \theta}}

\def \diam{\mathrm{diam}}

\def\EE{\mathbb{E}}

\def \fL {\mathfrak{l}}

\begin{abstract}
Let $\mathrm{SO}(3,\mathbb{R})$ be the 3D-rotation group equipped with the real-manifold topology and the normalized Haar measure $\mu$.
Resolving a problem by Breuillard and Green, we show that if $A \subseteq \mathrm{SO}(3,\mathbb{R})$ is an open subset with sufficiently small measure,
then 
$$ \mu(A^2) > 3.99 \mu(A).$$
We also show a more general result for the product of two sets, which can be seen as a Brunn--Minkowski-type inequality for sets with small measure in $\mathrm{SO}(3,\mathbb{R})$.
\end{abstract}

\maketitle

\section{Introduction}

\subsection{Results and backgrounds} Throughout, let $\Sot$ be the 3D-rotation group, or more precisely,
$$ \Sot := \{ Q \in M_{3,3}(\RR):  QQ^{T} = Q^{T}Q= I_3, \mathrm{det}(Q)=1\} $$
with $M_{3,3}(\RR)$ the set of $(3\times 3)$-matrices with real coefficients, $Q^T$ the transpose of $Q$, $I_3$ the identity $(3\times 3)$-matrix, and the group operation on $\Sot$ given by matrix multiplication.
 As a consequence, $\Sot$ can be identified with a Euclidean closed subset of $\RR^9$, which gives it a real-manifold topology. 
 It is well-known that the group $\Sot$ is compact and connected with respect to this topology. In particular, it has a normalized Haar measure $\mu$.  For $A,B \subseteq \Sot$, we are interested in their product set $ AB:= \{ab : a \in A, b \in B \}.$ We write $A^2$ instead of $AA$, and define the $k$-fold product $A^k$ similarly.

 The following question was described  in Green's list of open problem~\cite{Green}, where it was also attributed to discussions with Breuillard: 
 \begin{center}
      { \it When $A \subseteq \Sot $ is open and has sufficiently small measure, is $\mu(A^2) > 3.99\mu(A)$? }
 \end{center}
  It was also remarked in ~\cite{Green} that, if true, this will be the best possible, as seen by considering small neighbourhoods of a 1-dimensional subgroup. (Indeed, the construction in Section~\ref{sec: constructionandexamples} yields such open $A\subseteq \Sot$ with  $3.99\mu(A)< \mu(A)^2 < 4\mu(A)$ and $\mu(A)$ arbitrarily small, so we cannot replace $3.99$ by $4$.) In less precise form, the question traces back to the much earlier work of Henstock and Macbeath~\cite{HenstockMacbeath}, where they proposed the problem of determining minimal doubling in nonabelian locally compact groups. From this angle, $\Sot$ is of interest as it is the first nontrivial compact and connected case.

  Our main result answers the question by Breuillard and Green positively:
\begin{theorem}\label{thm: maingrowth}
For all $\epsilon>0$, there is $\delta>0$ such that if $A \subseteq \Sot$ is an open subset with $\mu(A)< \delta$, 
then 
$$ \mu(A^2) > (4-\epsilon) \mu(A).$$
\end{theorem}
Our proof, in fact,  gives the same conclusion for all compact semisimple Lie groups $G$ (Theorem~\ref{thm: mainBM2}), but the inequality is not sharp unless $G$ contains $\Sot$ as a closed subgroup.  Nevertheless, this provides a sharp contrast with compact and connected Lie group which is not semisimple. Here, the lower bound $\mu(A^2)\geq \min\{ 1, 2\mu(A)\}$
given by the Kemperman inequality cannot be improved. For open subsets of $\Sot$ with very small measure, Theorem~\ref{thm: maingrowth} improves the measure expansion gap by the first two authors~\cite{JT}, which says that there is a constant $\eta>0$ such that 
$$
\mu(A^2)\geq  \min\{1, 2\mu(A)+\eta\mu(A)|1-2\mu(A)|\}
$$
whenever $G$ is a compact semisimple Lie group, and $A \subseteq G$ is open. Note that construction in Section~\ref{sec: constructionandexamples} mentioned earlier also provides  open $A \subseteq \Sot$ with measure close to $1/2$ such that $2\mu(A)< \mu(A^2)< 2.01 \mu(A) $. Hence, the condition that $\mu(A)$ has sufficiently small measure is necessary, and Theorem~\ref{thm: maingrowth} does not provide information about general open $A \subseteq \Sot$.

Breuillard and Green studied the aforementioned question in connection with product theorems in groups of Lie type over finite fields~\cite{Helfgott08, BGT11, PS16} and results on approximate groups~\cite{Hrushovski, BGT}. While not stated explicitly, subgroups and neighborhoods are expected to be the explanation for small doubling in $\Sot$, just as they are in the abelian settings of additive combinatorics.  Thus, behind the question by Breuillard and Green is the challenge to generalize results in additive combinatorics to nonabelian groups. The result in this paper and the authors earlier work on the nonabelian Brunn--Minkowski~\cite{JTZ} are currently the only product-type theorems with sharp bounds. %; this partially confirm the above belief.

% Our theorem can be seen as a counterpart of product theorems in groups of Lie type over finite fields~\cite{Helfgott08, BGT11, PS16}. We are motivated by the fact that the bound in these results are not sharp, and the intuition that 

% In the past decades, there are numerous results suggesting results in additive combinatorics can be generalized into nonabelians settings. 
% Moreover, subgroups and neighborhoods are expected to be the explanation for small doubling in $\Sot$, just as they are in the abelian settings of additive combinatorics. Thus, behind the question by Breuillard and Green is the challenge to generalize results in additive combinatorics to nonabelian groups and to have sharp bounds. Here, $\Sot$ is of interest as it is the compact semisimple Lie group with smallest dimension.

 Recall that, for $\RR^n$ equipped with the usual Lebesgue measure $\lambda$,  the Brunn--Minkowski inequality tells us that if $X, Y\subseteq \RR^n$ are open, then 
$\lambda(X+Y)^{1/n} \geq \lambda(X)^{1/n} + \lambda(Y)^{1/n}  $ where we set $X+Y =\{x+y: x\in X, y\in Y\}$. 
 For $\Sot$, our proof of Theorem~\ref{thm: maingrowth} also yield the following more general asymmetric result, which can be seen as a Brunn--Minkowski type inequality for $\Sot$. 
\begin{theorem}\label{thm: mainBM}
 For all $\epsilon$ and $N$, there is $c =c(\epsilon, N)$ such that whenever  $A , B  \subseteq \mathrm{\Sot}$ are open, $0< \mu(A), \mu(B) < c$, and $\mu (A)/N  < \mu(B) <N \mu(A) $, we have
 \[
 \mu(AB)^{\frac{1}{2}+\epsilon}  \geq \mu(A)^{\frac{1}{2}+\epsilon}+ \mu(B)^{\frac{1}{2}+\epsilon}.
 \]
\end{theorem}
In Section~\ref{sec: constructionandexamples}, we will propose conjectural Brunn--Minkowski type inequalities for other compact connected Lie groups. The proofs of our theorems, in fact, provide a reduction of these results to the nonabelian Brunn--Minkowski conjecture for noncompact groups and a measure expansion gap result; see Remark~\ref{remark: section 10}.

We end this background discussion  proposing a conjectural strengthening of Theorem~\ref{thm: maingrowth}: 

\begin{conjecture}[Strong Breuillard--Green Conjecture]
If $A \subseteq \Sot$ is open,  then 
$$ \mu(A^2) \geq \min \{1, 4 \mu(A) (1-\mu(A))\}. $$
Moreover, if $\mu(A)< 1/2$, the equality happens if and only if $A$ is of the form
$$  \{ g\in \Sot: \angle (u, gu) < \arccos(1 -\mu(A)) \}  $$
with $u \in \RR^3$ a unit vector, $gu$ its image $g$-action, and  $\angle (u, gu)$ the angle between them.
\end{conjecture}
 We will discuss the construction behind this conjecture and generalizations to simple Lie groups with finite center in Section~\ref{sec: constructionandexamples}.

\subsection{Overview of the proof}

The central idea of the proofs of Theorem~\ref{thm: maingrowth} and Theorem~\ref{thm: mainBM} is to link the compact setting of $\Sot$ to the setting of noncompact Lie groups using ultraproduct and the Massicot--Wagner version~\cite{MW} of the so-called Hrushovski's Lie model theorem~\cite{Hrushovski}. The compact settings differs from the noncompact ones in that there is a useful nonabelian Brunn--Minkowski inequality in the latter; this was proven by the authors~\cite{JTZ} using the Iwasawa decomposition to reduce the problem to lower dimensions.
For a compact semisimple Lie group, the Iwasawa decomposition  returns the group itself.

For concreteness, towards a contradiction, let us assume that the conclusion of Theorem~\ref{thm: maingrowth} is false for $\epsilon=0.01$. Then there is a sequence $(A_n)$ of open subsets of $\Sot$ such that 
$$\mu((A_n)^2)< 3.99 \mu(A_n)  \quad \text{and}\quad  \lim_{n \to \infty} \mu(A_n) = 0.$$
Let $\mu_n$ be the normalization of $\mu$ on $A_n$ (i.e. setting $\mu_n(X) = \mu(X)/\mu(A_n)$ for measurable $X\subseteq \Sot$), we get 
\[
\mu_n(A_n)=1,\quad \mu_n((A_n)^2) < 3.99 \mu_n(A_n),\quad \text{and}\quad  \lim_{n \to \infty} \mu_n( \Sot  ) =\infty.
\]
Taking a ``suitable limit'' of the sequence of triples $( \Sot, A_n, \mu_n) $ we arrive at a triple $(G_\infty, A_\infty, \mu_\infty)$ where $G_\infty$ is a ``pseudo-compact''  group, $A_\infty \subseteq G_\infty$ is ``pseudo-open'', $\mu_\infty$ is a ``pseudo-Haar'' measure, and
$$\mu_\infty(A_\infty)=1,\quad  \mu_\infty((A_\infty)^2) < 3.99 \mu_\infty(A_\infty),\quad  \text{and}\quad  \mu_\infty( G_\infty  ) =\infty.$$
The suitable limit notion we use here is taking ultraproduct from model theory, and one can think of it as taking the average of $( \Sot, A_n, \mu_n) $. Notions like ``pseudo-Haar'' must also be carefully defined, but we will not do that in this overview.

Despite what the name ``pseudo-compact'' and ``pseudo-open'' might suggest, there is no automatic topological data on $G_\infty$. However, the measure-theoretic data on $G_\infty$ is enough to apply tools from the study of approximate groups to modify $A_\infty$ ``slightly'' and  construct a ``good'' surjective group homomorphism $$\pi: \langle A_\infty \rangle \to L$$  where $L$ is  noncompact, unimodular (i.e., left Haar measures are also right-invariance), and connected Lie group. Such $\pi$ is often called a Lie model.
We will postpone explaining this point for now and proceed with the argument. Using ideas from real/harmonic analysis, also to be revisited later, 
one shows that
$$ \lambda( X^2) < 3.99 \lambda(X)$$
with $X = \pi(A_\infty)$ and $\lambda$ a left (and hence right) Haar measure on $L$.

Let $d$ be the dimension of $L$, and $m$ the maximum dimension of a compact subgroup of $L$. The nonabelian Brunn--Minkowski Conjecture~\cite[Conjecture 1.4]{JTZ} predicts that $ \lambda(X^2) \geq 2^{d-m} \lambda (X).$
This is still not known in general, but we do know that  
$$ \lambda(X^2) \geq 2^{d-m -\lfloor (d-m)/3 \rfloor} \lambda (X).$$ Applying it into our case, we get $d-m<2$, which implies $d-m=1$ because $L$ is noncompact. In other words, $L$ has a compact subgroup $H$ of codimension $1$. By standard Lie-theoretic argument, we learn that $H$ must be a normal subgroup of $L$, and $L/H$ must be an isomorphic copy of $\RR$.

 Now, let $I \subseteq \RR$ be the interval $(0, 1)_\RR$. Then $I+I =\{ x+y : x, y \in I \}$ is the interval $(0, 2)_\RR$ which has exactly twice the Lebesgue measure of $I$. Let $\rho: \langle A_\infty \rangle \to \RR$ denotes the composition $\phi \circ \pi$ of $\pi: \langle A_\infty \rangle \to L$ with the quotient map $\phi: L \to \RR$, and set $B_\infty = \rho^{-1}(I)$. Then from the fact that $\pi$ is well-behaved, we learn that
 $$ \mu_\infty( (B_\infty)^2) = 2\mu_\infty(B_\infty).  $$
  As $G_\infty$  is the limit of copies of $\Sot$, and from the fact that $\pi: \langle A_\infty \rangle \to L$ is ``good'', we get  $B_n \subseteq \Sot$ with very small measure such that $$\mu(B_n^2) < (2+10^{-12}) \mu(B_n).$$
 This is known to be impossible by the measure expansion gap for semisimple Lie group developed by the first two authors in~\cite{JT}. This completes our proof via ``bootstrapping''.

\subsection*{Approximate groups and Lie models.} We now come back to an earlier point where we ``slightly'' modify $A_\infty$ and obtain a Lie model  $\pi: \langle A_\infty \rangle \to L$.

Conceptually, such a homomorphism reflects the expectation that  sets with small doubling in any setting are supposed to have origin in Lie groups.
 However, the small doubling condition $\mu_\infty(A_\infty^2) < 3.99 \mu_\infty(A_\infty)$ is not strong enough to guarantee the existence of such $\pi$ in the literal sense. One must go around this problem by first using a technique by Tao~\cite{Tao-expansion} to construct an approximate subgroup $S_\infty \subseteq G_\infty$ closely related to $A_\infty$.  (Recall that $S_\infty \subseteq G_\infty$ is called an approximate group if $\id_{G_\infty} \in S$, $S_\infty= S_\infty^{-1}$, and $S_\infty^2$ is covered by finitely many translates of $S_\infty$). 

Under an assumption called definable amenability, a version of the Hrushovski's Lie model theorem by Massicot and Wagner allows the construction of ``good'' surjective group homomorphism $\pi: \langle S_\infty \rangle \to L$ where $L$ is a locally compact group (but not yet a connected Lie group). Fortunately, one can modify $A_\infty$ and $S_\infty$ to make them ``pseudo-semialgebraic'' and arrange that the definability condition is satisfied.
Next, we use the Gleason--Yamabe theorem~\cite{Gleason,Yamabe} to obtain an open subgroup $L'$ of $L$, and a normal compact subgroup $K$ of $L'$ such that $L'/K$ is a connected Lie group. Replacing $S_\infty$ with $S^4_\infty \cap \pi^{-1} (L')$, the locally compact group $L$ with the Lie group $L'/K$, and $\pi$ with $\phi \circ \pi|_{ \langle S^4_\infty \cap \pi^{-1} (L') \rangle }$, we arrange that $L$ is a connected Lie group.

We cannot completely replace $A_\infty$ by $S_\infty$ because the latter might have doubling rate much larger than $3.99$ even though still bounded. In a mock version of the actual argument, we construct $A'_\infty \subseteq \langle S_\infty \rangle$ from $A_\infty$ and $S_\infty$ such that $\mu_\infty((A'_\infty)^2) < 3.99 \mu_\infty(A'_\infty)$. The set $A'_\infty$ is obtained by taking the intersection of $\langle S_\infty \rangle$ and a random translate of $A_\infty$ with respect to a carefully chosen probability measure.
Noting that $\langle A'_\infty \rangle =\langle S_\infty \rangle$ because of the connectedness of $L$, and we replace the original $A_\infty$ by this $A'_\infty$. The actual argument is slightly more complicated than the above mock version. Instead of getting $A'_\infty$ as above, we get two sets $\Tilde{A}_\infty$ and $\Tilde{A}'_\infty$ satisfying an asymmetric small doubling condition in the form of a Brunn--Minkowski-type inequality. Therefore, even if we are only interested in Theorem~\ref{thm: maingrowth}, the proof required essentially also handle  Theorem~\ref{thm: mainBM}.

We end this part remarking that the argument via Hrushovski's Lie model theorem has deep roots within model theory. The origin of such approach traces back to Robinson's nonstandard-analysis~\cite{Robinson} and is related to other results like van den Dries--Wilkie reproof of Gromov theorem using ultraproduct~\cite{VanWilkie}, the body of work surrounding Pillay's Conjecture~\cite{Pillay}, Goldbring's proof of Hilbert's 5th problem for local groups~\cite{Goldbring}.

\subsection*{Density of definable sets over Lie models.} We next briefly indicate the ideas from real/harmonic analysis used in deducing $ \lambda( X^2) < 3.99 \lambda(X)$ when we have the Lie model $\pi: \langle A \rangle \to L$, with $X = \pi(A_\infty)$ and $\lambda$ a Haar measure on $L$.

If $\langle A_\infty \rangle$ were a locally compact group, the pseudo-Haar measure $\mu_\infty$ and the Haar measure $\lambda$ on $L$ can be linked together by a Haar measure $\nu$ on $\ker \pi$ and a quotient integral formula. This allows us to relate  $\mu_\infty(A_\infty)$ and $\mu_\infty(A^2_\infty)$ with the measure of the images $\lambda(X)$ and $\lambda(X^2)$ using the density function $f_A(x) = \nu ( \ker \pi \cap x^{-1}A  )$. The argument then proceeds using a ``spillover'' technique  in~\cite{JT} and \cite{JTZ}.

There is no similar measure on the kernel $\ker \pi$ of our Lie model. The Radon--Nikodym theorem does imply there is a density functions defined almost everywhere linking $\mu_\infty$ and $\lambda$. This is still not good enough for us for the following reason.  We  need to study the relationship between the density functions $f_{A_\infty}$ and $f_{A^2_\infty}$. As $A^2_\infty$ is generally a uncountable union of translates of $A_\infty$, the bad behavior at a null-set of points might contribute too much in a product, making such relationship unclear.

In our approach, we approximate $f_{A_\infty}$ by a family of  better behaved functions $(f^\epsilon_{A_\infty})_{\epsilon \in \RR^{>0}}$, where $f^\epsilon_{A_\infty}$
 is obtained  by considering average behavior of  $f_{A_\infty}$ in a suitable $\epsilon$-ball around the point under consideration.  The Lebesgue differentiation theorem for the Lie group $L$ implies the convergence of $(f^\epsilon_{A_\infty})_{\epsilon \in \RR^{>0}}$ to $f^\epsilon_{A_\infty}$ almost everywhere. We also do the same for $f_{A^2_\infty}$. It turns out that $(f^\epsilon_{A_\infty})_{\epsilon \in \RR^{>0}}$ are well-behaved enough to replace the role played by usual density function and allow us to obtain the desired conclusion.

  We remark that the Lebesgue differentiation theorem over manifolds we used is a consequence of the weak-$L^1$ estimates for the Hardy--Littlewood maximal function from harmonic analysis.

\subsection{Structure of the paper}
The paper is organized as follows. Section~2 includes some facts about Haar measures, Riemannian metrics, and Lie groups, which will be used in the subsequent parts of the paper.  In Section 3, we construct examples in $\Sot$ to show that for every $c\in(0,1)$ there is a set $A$ with $\mu(A)=c$ with $\mu(A^2)=4\mu(A)(1-\mu(A))$. We also generalize this construction to other simple Lie groups and make some more general conjectures. 

Section~4 allows us to find sets with doubling doubling smaller than $4-\varepsilon$ in the ultraproduct group. This step corresponds to the ``taking limit'' step in the overview. In Section 5, we will find a definably amenable open approximate group that is commensurable to the original set with with small doubling in the ultraproduct group. In Section 6, we will find a subgroup of the ultraproduct group and then go down to a connected Lie model via Massicot--Wagner version of the Hrushovski's Lie model theorem. 

In Section 7, we will reconstruct sets with doubling smaller than $4-\varepsilon$ in the subgroup of the ultraproduct group obtained in Section 6 from the approximate group. In Section 8, we will introduce a density function to connect doubling of sets in the ultraproduct group and doubling of their projections in the Lie model. Section 9 allows us to produce a set with doubling smaller than $4-\varepsilon$ in the connected Lie group using density function. In Section 10, we prove the main theorem by pulling back sets with doubling close to $2$ in the Lie model and obtaining a contradiction to the measure growth gaps.

\subsection{Notation and convention} From now on,  $k$ and $l$ are in the ordered ring $\ZZ$ of integers, $m$ and $n$ are in the ordered semiring $\NN =\{0, 1, \ldots\}$ of natural numbers. By a constant, we mean an element in the positive cone $\RR^{>0}$ of the ordered ring $\RR$ of real numbers. 

We will let $G$ with possible decorations denote a multiplicative group possibly with more data (topology, differentiable structures, etc). By a measure on $G$, we mean a nonnegative measure $\mu: \Sigma \to \RR^{\geq 0}$ on a $\sigma$-algebra $\Sigma$ of subsets of $G$. If set $A$ is in this $\sigma$-algebra $\Sigma$, we will say that $A$ is measurable with respect to $\mu$. 

By a locally compact group, we means a topological group which is locally compact as a topological space. Similar conventions apply to compact groups, connected groups, etc.

A measure on a locally compact group is always assumed to be a Haar measure. By a measurable subset of a locally compact group, we mean a set measurable with respect to some Haar measure, equivalently, with respect to the up-to-constant unique complete Haar measure. See Section~\ref{sec: Prelimmeasurelocallycompact} for more details.

Given a set $A$, we write $A^n$ for the $n$-fold product set of $A$, that is $\{a_1\cdots a_n\mid a_1,\dots,a_n\in A\}$. We write $A^{[n]}$ for the $n$-dimensional Cartesian product of $A$, but we still write $\RR^n$ for the $n$-dimensional Euclidean space for the notational convention. 

We use the asymptotic notation from harmonic analysis. That is, $f\lesssim g$ means $f=O(g)$, and $f\sim g$ if $f\lesssim g$ and $g\lesssim f$. We write $f=O_{\alpha_1,\dots,\alpha_n}(g)$, $f\lesssim_{\alpha_1,\dots,\alpha_n} g$, and $f\sim_{\alpha_1,\dots,\alpha_n} g$ when the hidden constant(s) depends  on  $\alpha_1,\dots,\alpha_n$.

\section{Preliminaries}

We recall here a number of standard facts about measures, Riemannian metrics, locally compact groups, and Lie groups. Advanced concepts more directly related to the main argument will be included later on.

\subsection{Measures and locally compact groups} 
\label{sec: Prelimmeasurelocallycompact}

Recall that a {\bf premeasure} $\mu_0$ on an ambient set $\Omega$ is a nonnegative real-valued function on a Boolean algebra  (closed under finite union and taking relative complement) of subsets of $\Omega$ such that the following holds:
\begin{enumerate}
    \item [(PM1)] $\mu_0(\emptyset) =0$
    \item [(PM2)] ($\sigma$-additivity) If $(A_n)$ is a sequence of sets such that $\mu_0(A_n)$  for each $n$ and $\mu_0(\bigcup n A_n)   $ are
     are well defined, then 
    $$ \mu_0\Big(\bigcup_n A_n\Big)= \sum_{i=1}^n \mu_0(A_i). $$
\end{enumerate}
In other words, the premeasure $\mu_0$ behaves like a measure, except that the collection of set it applies to might not be a $\sigma$-algebra (also closed under countable union). We will later need the following fact, which allows us to construct measures from premeasures:

\begin{fact}[Carath\'eodory's extension theorem] \label{Caratheodory}
Suppose $\mu_0$ is a premeasure on an ambient set $\Omega$. Then there is a measure on $\Omega$ extending $\mu_0$.
\end{fact}

We say that $\mu$ is {\bf complete} if every subset of a $\mu$-null set is measurable. It is well known that there is a smallest complete measure extending $\mu$, which we will refer to as the {\bf completion} of $\mu$.

Let $\Omega$ be a topological space. A measure $\mu$ on $\Omega$ is a {\bf Borel measure} if the following hold:
\begin{enumerate}
    \item[(BM1)] Every Borel subset (i.e. member of the $\sigma$-algebra generated by open sets) of $\Omega$ is $\mu$-measurable.
    \item[(BM2)] $\mu$ is the completion of its restriction to the $\sigma$-algebra of Borel subsets of $\Omega$.
\end{enumerate}

Again $\Omega$ is a topological space. A Borel measure $\mu$ on $\Omega$ is an {\bf outer Radon measure} if we have the following
\begin{enumerate}
    \item[(RM1)] (Locally finite) Every $x \in \Omega$ has an open neighborhood with finite measure.
    \item[(RM2)](Outer regularity)  The measure of a Borel $A \subseteq \Omega$ is the infimum of the measure of the open subsets of $\Omega$ containing $A$.
    \item[(RM3)](Inner regularity of open set) The measure of an open $U \subseteq \Omega$ is the supremum of the measure of the compact subsets of $U$.
   
\end{enumerate}

Suppose $G$ is a group and $H \leq G$. A measure $\mu$ on the left-coset space $G/H$ is {\bf left-invariant} if for all measurable $A \subseteq G/H$ and $g\in G$, we have $\mu(A)=\mu(gA)$.

A {\bf locally compact group} $G$ is a group equipped with a locally compact and Hausdorff topology on its underlying set such that multiplication and inversion are continuous.
It is easy to see that when $H$ is a compact subgroup of $G$, the left-cosets space $G/H$ equipped with the quotient topology  is also a locally compact and Haussdorff topological space. We have the following fact~\cite[Chapter~1]{Harmonicanalysis}:

\begin{fact} \label{fact: Haarmeasurenew}
Suppose $G$ is a locally compact group, and $H \leq G$ is compact. Then there is a left-invariant  nonzero Radon measure on the left-cosets space $G/H$. Moreover, any two such measures differ by a constant. 
\end{fact}

In particular, with $H= \{1_G\}$, the locally compact group $G$ can be equipped with a left-invariant complete nonzero Radon measure, which is called a {\bf left Haar measure}. Any two left Haar measures differs only by a positive constant. If a left Haar measure $\mu$ on $G$ is right-invariant (i.e., $\mu(A) = \mu(Ag)$ for all measurable $A \subseteq G$ and $g\in G$, we call it a {\bf  Haar measure}. If one (equivalently, all) Haar measures on $G$ is right-invariant, we say that $G$ is {\bf unimodular}. The additive group $\RR^d$ with the Euclidean topology and compacts group are, in particular, unimodular.

Below are some other facts about the Radon measure in Fact~\ref{fact: Haarmeasurenew} that we will use.

\begin{fact} \label{Fact: FurtherpropertiesHaar}
Suppose $G$ is a locally compact group, and $H \leq G$ is compact. Suppose $\mu$ is a left-invariant  complete nonzero Radon measure on the left-cosets space $G/H$. Then we have the following:
\begin{enumerate}
    \item Open sets have positive measures.
    \item Compact sets have finite measures.
\end{enumerate}
\end{fact}

 \subsection{Riemannian metrics and Lie groups}
 \label{sec: PrelimRiemannianmetricandLiegroup}
 The material in this section is standard and can be found in any book of Riemannian geometry; see~\cite{Riemann}, for example.

Let $M$ be a smooth manifold, a {\bf Riemannian metric} $g$ on $M$ assigns in a smooth fashion to each point $p\in M$ a positive definite inner product $g_p$ on the tangent space $T_pM$ of $M$ at the point  $p$. Here, {\bf positive definite} means $g_p(v,v) \geq 0$ for each $v\in T_pM$ and the equality holds if and only if $v=0$. A {\bf Riemannian manifold} $(M, g)$ is a smooth manifold together with a Riemannian metric $g$ on $M$.

Suppose $(M, g)$ is a Riemannian manifold with dimension $n$. Let $U \subseteq M$ be a coordinate patch, and  $x: U \to \RR^n, p \mapsto( x_1(p), \ldots, x_n(p))$ the coordinate function. If $K \subseteq U$ is compact, we define 
$$ \mathrm{Vol}_g(K) = \int_{z\in x(K)} \sqrt{G \circ x^{-1}}(z) d\lambda . $$
where $G = \det( g_{ij})_{(i,j) \in [n]\times [n]}$, $g_{ij} = g(\partial_i, \partial_j)  $, and $\lambda$ is a Lebesgue measure on $\RR^n$. Recall that $A \subseteq M$ is measurable if for each coordinate patch $U \subseteq M$ and coordinate function $x: U \to \RR^n$, the image $x(A \cap U)$ is measurable. For a measurable $A \subseteq M$, we can define its volume $\mathrm{Vol}_g(A)$ to be the supremum of volume of finite union of compact sets each contained in a coordinate patch. We have the following fact:
\begin{fact}
     $\mathrm{Vol}_g$ is a complete measure on $M$ compatible with the topology on $M$. Moreover, $\mathrm{Vol}$ is independent of a choice of coordinate patch.
\end{fact}

Suppose $(M, g)$ is a Riemannian manifold with dimension $n$. Let $U \subseteq M$ be a coordinate patch, and  $x: U \to \RR^n, p \mapsto( x_1(p), \ldots, x_n(p))$ the coordinate function. Define length $L_g$ of an admissible curve.

 It is a well-known fact that in a connected Riemannian manifold, every two points $p$ and $q$ are connected by an admissible curve. We define 
 $$ d_g(p,q) = \inf \{ \gamma: \gamma \text{ is an admissible curve linking } p \text{ and } q  \}. $$
 
\begin{fact}\label{fact: curve}
     $d_g$ is a metric compatible with the topology on $M$. 
\end{fact}

If $d$ is a metric on $M$ such that $d=d_g$ for some Riemannian metric $g$, we say $d$ is the distance function {\bf induced by a Riemannian metric} $g$. The following fact gives us the existence of an invariant Riemannian metric.
\begin{fact}
If a Lie group $G$ acts smoothly and transitively on a smooth manifold $M$ with compact isotropy groups, then there exists a $G$-invariant Riemannian metric on $M$.
\end{fact}

Suppose $(M,g)$ is an $n$-dimensional Riemannian manifold. The {\bf Ricci curvature} is the covariant $2$-tensor field defined as the trace of the
curvature endomorphism on its first and last indices, and the {\bf scalar curvature} is the function $S$ defined as the trace of the Ricci curvature. 
The next fact shows that the volume of a infinitesimal ball is controlled by the scalar curvature.
\begin{fact}\label{fact: ball}
Let $(M, g)$ be an  $n$-dimensional Riemannian manifold with constant scalar curvature $S$. Then when $r\to 0$, a ball of radius $r$ at the identity point has volume 
\[
\frac{1}{n}\alpha_{n-1}r^n\left(1-\frac{Sr^2}{6(n+2)}+O(r^3)\right),
\]
where
\[
\alpha_n=
\begin{cases}
\frac{2^{2k+1}\pi^kk!}{(2k)!}\qquad &\text{when } n=2k, k\in\ZZ, \\
\frac{2\pi^{k+1}}{k!}, &\text{when } n=2k+1, k\in\ZZ.
\end{cases}
\]
\end{fact}

\section{Constructions and conjectures} 
\label{sec: constructionandexamples}

In this section, we will go through in more details the earlier mentioned example of open $A \subseteq \Sot$ with measure smaller than $1/2$ such that 
$$\mu(A^2)= 4\mu(A) -4( \mu(A))^2 < 4\mu(A).$$ 
We will state the relevant conjectures and generalizations for semisimple Lie groups.

Before stating the example, we need a couple of lemmas about some basic properties of $\Sot$. Let $u$, $v$, and $w$ range over unit vectors in $\RR^3$, and let $\alpha(u,v):= \arccos(u\cdot v)$ be the angle between $u$ and $v$. Let $\phi$ and $\theta$ range over $\RR$. Use  $R_u^\phi$ to denote the counter-clockwise rotation with signed angle $\phi$ where the axis and the positive direction (under the right-hand rule) is specified by the unit vector $u$. We first recall Euler's rotation theorem:

\begin{fact} \label{fact: axisangle}
    Each nontrivial element  $g \in \Sot$ is of the form $R_u^\phi$ with $\phi \in (0, 2\pi)$. Moreover, the set of elements in $\RR^3$ fixed by such $g$ is  $\mathrm{span}(u) =\{\lambda u \mid \lambda \in \RR\}$.    
\end{fact}

 The following lemma is essentially a variation of the above fact.

\begin{lemma} \label{lem: good representation}
Let $u$ be a fixed unit vector. Then every $g\in \Sot$ is of the form $R^\theta_v R^\phi_u $ with $v$ a unit vector orthogonal to $u$. Likewise, every $g\in \Sot$ is of the form $R^{\phi'}_u R^{\theta'}_{w} $ with $w$ a unit vector orthogonal to $u$.
\end{lemma}

\begin{proof}
We prove the first assertion. Choose $v$ to be the normal vector of $\mathrm{span}(u, g(u))$. Then $v$ is orthogonal to $u$, and there is $\theta$ such that $g(u)  = R^{\theta}_v (u) $. Now, $g^{-1} R^\theta_v$ fixes $u$, so by Fact~\ref{fact: axisangle}, $g^{-1} R^\theta_v = R^{-\phi}_u$ for some $\phi \in \RR$. Thus, $R^\theta_v R^\phi_u$.

The second assertion can be obtained by applying the first assertion to $g^{-1}$. 
\end{proof}

We will need the following inequality:

\begin{lemma}\label{lem: example 2} Let $u$ be a fixed unit vector. 
    Suppose $g_1, g_2 \in \mathrm{SO_3}(\RR)$ are such that $\alpha(u, g_i(u)) \leq  \pi/2$ for $i \in {1, 2}$. Then we have $\alpha(z, g_1g_2(z)) \leq \alpha(z, g_1(z)) + \alpha(z, g_2(z)) .$
\end{lemma}

\begin{proof}
Applying Lemma~\ref{lem: good representation}, we can write $g_2$ as $R^{\theta_2}_vR^{\phi_2}_u $ and $g_1$ as $R^{\phi_1}_uR^{\theta_1}_w  $ with $v$ and $w$ orthogonal to $u$. 
It is easy to see that 
$$ \alpha(u, g_1(u)) = \alpha(u, R^{\theta_1}_w(u)) \text{ and } \alpha(u, g_2(u)) = \alpha(u, R^{\theta_2}_v(u)).$$
On the other hand, 
$$ \alpha(u, g_1g_2(u)) = \alpha(u, (R^{\theta_1}_wR^{\theta_2}_v)(u)). $$ 
The triangle inequality in term of angles gives us 
$$\alpha(u, (R^{\theta_1}_wR^{\theta_2}_v)(u)) \leq \alpha(u, R^{\theta_2}_v(u)) +  \alpha(R^{\theta_2}_v(u), (R^{\theta_1}_wR^{\theta_2}_v)(u)).   $$
As $w$ is orthogonal to $u$ and possibly not to to $R^{\theta_2}_v(u)$, we have 
$$\alpha(R^{\theta_2}_v(u), (R^{\theta_1}_wR^{\theta_2}_v)(u)) \leq \alpha(u, R^{\theta_1}_w(u)). $$
The desired conclusion follows.
\end{proof}

\begin{lemma} \label{lem: example}
Let $u$ be a unit vector, $\theta\in (0,\pi/2)$, and $A = \{g \in \Sot \mid \alpha(u, g(u)) \leq \pi/2 -\varepsilon\} $. Then 
$$A^2 = \{g \in \Sot \mid \alpha(u, g(u)) \leq \pi -2\varepsilon \}. $$
\end{lemma}
\begin{proof}
The inclusion $A^2 \subseteq \{g \in \Sot \mid \alpha(u, g(u)) \leq \pi -2\varepsilon \}$ is immediate from Lemma~\ref{lem: example 2}. Now suppose $g \in \Sot$ satisfies $\alpha(u, g(u)) \leq \pi -2\varepsilon$. Then, Lemma~\ref{lem: good representation} yields $g= R^\theta_v R^\phi_u $ with $v$ orthogonal to $u$ and $\theta \in [2\varepsilon- \pi, \pi -2\varepsilon]$. We can rewrite $g = R^{\theta/2}_v (R^{\theta/2}_vR^\phi_u) $.
The other inclusion follows.
\end{proof}

We need the following facts about the unit sphere:
\begin{fact}
    Let $S^2 = \{ u \in \RR^3 : \|u\| =1 \}$. Then there is a complete Radon measure $\nu$ of $S^2$ satisfying the following conditions:
    \begin{enumerate}
        \item $\nu$ is invariant with respect to the left action by $\Sot$ and $\nu(S^2)=1$.
        \item If $X \subseteq S^2$ is the set $\{ u \in \RR^3 : \theta_1 \leq \theta(u) < \theta_2, \phi_1 , \leq \phi(u), < \phi_2 \}$, then $\nu(X)$ is given by the Riemann integral
        $$ \int_{\theta_1}^{\theta_2} \int_{\phi_1}^{\phi_2} \sin(\theta)\d\theta \d\phi.$$
               
    \end{enumerate}
\end{fact}

The next proposition is our construction. This can be seen as a generalization of the example given in~\cite{JT} by the first two authors. 
%We now prove the main statement of this section:
\begin{proposition} \label{prop: example}
Suppose an open $A \subseteq \Sot$ is of  the form 
$  \{ g\in \Sot : \angle (u, gu) < \theta \} $
with $u \in \RR^3$ a unit vector and $\theta \in (0, \pi/2] $. Then we have
$$\mu(A^2) = 4\mu(A)(1 - \mu(A)).$$
In particular, $\mu(A^2) < 4\mu(A)$ and the measure of $A$ can be arbitrarily small.
\end{proposition}
\begin{proof}
 Let $u$ and $A$ be as in Lemma~\ref{lem: example}. 
 Recall that the group $\Sot$ acts transitively on the 2-sphere $S^2$ consisting of all unit vectors in $\RR^3$. Let $ T \leq \Sot$  be the stabilizer of $u$. Then, $\Sot/T$ can be identified with $S^2$. With $\pi A$ and $\pi A^2$ be the projections of $A$ and $A^2$ in $\Sot/T$ respectively, we have 
 \[
 \pi A =\Big\{v \mid \alpha(u,v) \leq \frac{\pi}{2}-\varepsilon\Big\} \text{ and } \pi A^2 =\{v \mid \alpha(u,v) \leq \pi-2\varepsilon\}. 
 \]
 Let $\nu$ be the Radon measure induced by $\mu$ on $\Sot/T$ via the quotient integral formula.
From the way we construct $A$ and Lemma~\ref{lem: example}, we have $A=AT$ and $A^2=A^2T$.  Hence, $\mu(A)= \nu(\pi A)$ and $\mu(A^2)=\nu(\pi A^2)$. Finally, note that $\nu$ is the normalized Euclidean measure, so an obvious computation yields the desired conclusion.
\end{proof}

In light of Proposition~\ref{prop: example}, it is natural to generalize the aforementioned Breuillard--Green conjecture to sets with all possible measures in $\Sot$. 

\begin{conjecture}[Strong Breuillard--Green Conjecture]
Suppose $A \subseteq \Sot$ is open. Then 
$$  \mu(A^2) \geq \min \{1, 4 \mu(A) (1-\mu(A))\}. $$
Moreover, if $\mu(A)< 1/2$, the equality happens if and only if $A$ is of the form
$$  \{ g\in \Sot: \angle (u, gu) < \arccos(1 -\mu(A)) \}  $$
where $u \in \RR^3$ is a unit vector.
\end{conjecture}

The aforementioned Breuillard--Green Conjecture for $\Sot$ also has a natural generalization to all compact simple Lie groups. 

\begin{conjecture}[Breuillard--Green Conjecture for compact simple Lie groups]\label{conj: generalization}
Let $G$ be a compact simple Lie group with dimension $d$ equipped with a normalized Haar measure $\mu$, and the dimension of the maximal compact proper subgroup is $m$. Then for every $\varepsilon>0$, every compact $A\subseteq G$ with sufficiently small measures, 
\[
\mu(A^2)>(2^{d-m}-\varepsilon)\mu(A).
\]
\end{conjecture}

We now discuss a more general construction. We start with some fact about compact simple Lie group and symmetric space~\cite{Riemann}.

\begin{fact}\label{fact: compact group metric}
Suppose $G$ is a compact simple Lie group, then up to a constant, there is a unique bi-invariant Riemannian metric on $G$ inducing a distance function $\Tilde{d}$ on $G$.
Moreover, let  $H \leq G$ be a closed and hence compact subgroup of $G$. Then there is a unique left-invariant Riemannian metric on $G/H$ such that if $d$ is the distance function it induces on $G/H$, then 
$$ d(g_1H, g_2H): = \EE_{h \in H} \Tilde{d}(g_1h, g_2 h).$$
The Riemannian metric $d$ has positive constant scalar curvature.
\end{fact}

The following example shows that Conjecture~\ref{conj: generalization} if true, the $2^{d-m}$ factor is also sharp, and the $-\varepsilon$ term is necessary. 
\begin{proposition}\label{prop: general example}
   Let $G$ be a compact simple Lie group with dimension $d$ equipped with a normalized Haar measure $\mu$. Let $m$ be the dimension of a maximal compact proper subgroup. Then  for sufficiently $c>0$ there is $A\subseteq G$ with $\mu(A)=c$ and
   \[
   \mu(A^2)<2^{d-m}\mu(A).
   \]
\end{proposition}
\begin{proof}
Let $H\leq G$ be a maximal proper compact subgroup. Then $\dim H=m$. 
By Fact~\ref{fact: compact group metric}, $G/H$ is a symmetric space and hence it has a unique $G$-invariant Riemannian metric, and this metric induces a volume measure $\lambda$ on $G/H$. 

Note that the projection measure of $\mu$ on $G/H$ equals to $\lambda$ up to a constant scalar as the projection measure of $\mu$ on $G/H$ is also a $G$-invariant Borel measure. Let $B_r$ be an open ball of radius $r$ in $G/H$, and $D_r=\pi^{-1}(B_r)$ where $\pi:G\to G/H$ is the projection map. We claim that $D_r^2\subseteq D_{2r}$. 

To see the claim, let $g_1$ and $g_2$ be two arbitrary elements in $D_{r}$, and let $\pi g_1$ and $\pi g_2$ be the projections of $g_1$ and $g_2$ in $G/H$. For $i=1,2$, let $\gamma_i$ be the geodesic curve connecting $\pi g_i$ and the identity in $G/H$, and the length of each $\gamma_i$ is strictly smaller than $r$ by the choice of $g_i$ and Fact~\ref{fact: curve}. As the metric on $G/H$ is $G$-invariant, $(\pi g_1) \gamma_2$ has the same length as $\gamma_2$. Now let $\gamma$ be the curve formed by $(\pi g_1) \gamma_2$ after $\gamma_1$, then it is a curve connecting $\pi g_1g_2$ and the identity, and has length strictly smaller than $2r$. Thus $g_1g_2\in D_{2r}$ and hence $D_r^2\subseteq D_{2r}$. 

Using Fact~\ref{fact: ball} when $r$ is sufficiently small, we have
\[
\frac{\mu(D_r^2)}{\mu(D_r)}\leq \frac{\mu(D_{2r})}{\mu(D_r)}=\frac{\lambda(B_{2r})}{\lambda(B_r)}\to 2^{d-m} \quad(\text{as }r\to 0). 
\]
By Fact~\ref{fact: compact group metric} the metric on $G/H$ has a positive constant scalar curvature. Thus using Fact~\ref{fact: ball} we have $\frac{\mu(D_r^2)}{\mu(D_r)}<2^{d-m}$ when $r$ is sufficiently small. 
\end{proof}

One may compare Conjecture~\ref{conj: generalization} with the recently developed Brunn--Minkowski phenomenon by the authors~\cite{JTZ}:

\begin{fact}[Symmetric Brunn--Minkowski for simple Lie groups with finite center]\label{fact: symm BM}
Let $G$ be a simple Lie group with finite center, and $\mu$ a Haar measure on $G$. Let $d$ be the topological dimension of $G$, and $m$ be the dimension of a maximal compact subgroup of $G$. Then for every compact $A\subseteq G$,
\[
\mu(A^2)\geq 2^{d-m}\mu(A). 
\]
\end{fact}

For a more general Brunn--Minkowski result we refer to~\cite{JTZ}. We also remark that the Brunn--Minkowski inequality became trivial when the ambient group is compact, as in this case $m=d$ in Fact~\ref{fact: symm BM}, while in Conjecture~\ref{conj: generalization} we have $m<d$.

Now we look at a fact in noncompact simple Lie groups with a finite center~\cite{hilgert}. 

\begin{fact}\label{fact: noncompact}
Suppose $G$ is a noncompact simple Lie group with a finite center, and let $H\leq G$ be a closed and connected compact subgroup of $G$ having maximal dimension.
Then there is a unique left-invariant Riemannian metric on $G/H$. This Riemannian metric has negative constant scalar curvature.
\end{fact}

We will refer to the metric obtained in Facts~\ref{fact: compact group metric} and \ref{fact: noncompact} as the canonical metric on $G/H$. We now state a  general conjecture:

\begin{conjecture}[Minimal measure doubling Conjecture] \label{conj: minimal measure doubling}
Suppose $G$ is a simple Lie group equipped with a Haar measure $\mu$, and $A\subseteq G$ is open. Let $H \leq G$ be a proper compact connected subgroup of $G$ with maximal dimension, $\pi: G \to G/H$ is the quotient map, $B_r \subseteq G/H$ is the ball with radius $r$ centered at the coset $H$, and $B_{2r}$ is defined similarly. Assuming $\mu(A) = \mu(\pi^{-1}(B_r))$, we have 
$$ \mu(A^2) \geq \min\{ \mu(\pi^{-1}(B_{2r})), \mu(G)\}.$$
Moreover, the equality happens if an only if $A$ is a conjugate of $\pi^{-1}(B_r)$.
\end{conjecture}

We will next deduce a number of consequences of the above conjecture.

\begin{proposition} Assuming the minimal measure doubling Conjecture holds. Then we have the following:
\begin{enumerate}
    \item The Strong Breuillard--Green conjecture holds.
    \item If $A\subseteq \mathrm{SL}(2, \RR)$ is open, then there is a Haar measure $\mu$ that
    $$ \mu(A^2) \geq 4\mu(A)( 1+\mu(A)) $$
    Moreover, consider the action of $\mathrm{SL}(2, \RR)$ on the upper half plane by linear fractional transformation. The equality happens if and only if $A$ is of the form $$ \{ g\in \mathrm{SL}(2, \RR) : d( p, gp)<r  \}.$$ 
\item Fact~\ref{fact: symm BM} holds.
    \end{enumerate}
\end{proposition}
\begin{proof}
Statement (1) is clear from our construction, as well as the uniqueness of the $G$-invariant metric by Fact~\ref{fact: compact group metric}. 
Similarly, Statement (3) follows from Fact~\ref{fact: noncompact}, the fact that the scalar curvature is negative, and Fact~\ref{fact: ball} when the radius of ball approaches to $0$. We will only show (2) here, which can also be viewed as the first non-trivial case of Conjecture~\ref{conj: minimal measure doubling} for noncompact groups. 

Let $H\cong\mathrm{SO}(2,\RR)$ be a maximal compact subgroup of $G:=\mathrm{SL}(2,\RR)$. Then $G/H$ is isometric to the hyperbolic plane $\mathbb{H}_2$. Choose a Haar measure $\mu$ on $G$ so that under the induced metric  a ball of radius $r$ in $\mathbb{H}_2$ has volume
\[
\int_{0}^r\sinh(t)\d t. 
\]
Let $\pi:G\to G/H$ be the quotient map and $A=\pi^{-1}(B_r)$. Using the same proof as in Proposition~\ref{prop: general example} together with Fact~\ref{fact: noncompact} it is easy to see that $A^2\subseteq \pi^{-1}(B_{2r})$. Thus
\[
\mu(A^2)\leq \cosh(2r)-1=2\cosh(r)^2-2=4\mu(A)(1+\mu(A))
\]
as desired. 
\end{proof}

\section{Small growth in ultraproducts}

In this section, we will start from a decreasing size sequence of small growth sets in $\Sot$ and construct infinitesimal small growth set in an ultraproduct of $\Sot$.  We will treat more generally ultraproducts of compact groups. As it was mentioned earlier in the introduction, our proof necessitate considering the asymmetric problem for products of two sets.

Let $G$ be a unimodular locally compact group, and $\mu$ a complete Haar measure. Suppose $A,B \subseteq G$ are such that $A$, $B$, and $AB$  are measurable, and $0<\mu(A), \mu(B), \mu(AB)< \infty$.
We define the {\bf Brunn--Minkowski growth } $\BM(A,B)$ to be the unique $r >0$ such that 
$$ \ \mu(AB)^{\frac{1}{r}} = \mu(A)^{\frac{1}{r}} +\mu(B)^{\frac{1}{r}}.$$
In particular, with $A=B$, we have $\BM(A,A)\leq r$ if and only if $\mu(A^2)\leq  2^r \mu(A)$. Note that this definitions is independent of the choice of the complete Haar measure $\mu$ due to the up-to-constant uniqueness of complete Haar measure (Fact~\ref{fact: Haarmeasurenew}(6)). At first sight, this definition seem to also work in then nonunimodular setting with left Haar measure. However, the correct definition for nonunimodular locally compact groups involves both the left and right Haar measure; see~\cite{JTZ} for details. We will not discuss this issue further here as it will not be useful for the current purpose.

We now recall some element of the theory of ultrafilters and ultraproducts from logic. For a systematic treatment of ultrafilter and applications, see~\cite{Isaacbook}. A {\bf nonprincipal ultrafilter} $\ult$ on $\NN$ is a collection of infinite subsets of $\NN$ which satisfy the following conditions:
\begin{enumerate}
    \item[(NU1)] If $I, J \in \ult$, then $I \cap J \in \ult$.
    \item[(NU2)] If $I \in \ult$ and $I \subseteq J$, then $J \in \ult$.
    \item[(NU3)] For all $I \subseteq \NN$, either $I$ or $\NN \setminus I$ is in $\ult$. 
\end{enumerate}
One can think of a nonprincipal ultrafilter on $\NN$ as choosing a notion of ``almost everywhere/almost every'' on  $\NN$, where $I \in \ult$ if and only if ``almost every'' natural number $n$ is in $I$. 
This can be made precise by considering the finitely additive measure
$$ \cP(\NN) \to \{0, 1\}, \quad  I \mapsto 
\begin{cases}
1 & \text{ if } I \in \ult,\\
0 & \text{ if } I \notin \ult.
\end{cases} $$
Note that this assignment is not $\sigma$-additive, and hence not a measure. Indeed, for each $n$, the set $\{n\}$ is in $\ult$ as it is finite, but their union $\NN$ is not in $\ult$ being cofinite. 
Nevertheless, this is a very useful heuristic, and we will facilitate it further. For a fixed nonprincipal ultrafilter $\ult$, when $P$ is a property for natural number, we write a.e.\! $n$ satisfies $P$ if  $\{ n : P(n) \text{ holds}\}$ is in $\ult$. The existence of a non-principal ultrafilter depends on the axiom of choice. However, the truth of our theorem is independent of the axiom of choice by a Shoenfied's absoluteness argument.

Throughout, we fix an ambient set $V$ which contains all the mathematical object we care about; for instance, we can choose $V$ to be an initial segment of the set-theoretic universe. Fix a nonprincipal ultrafilter $\ult$ on $\NN$.
Two sequences $(a_n)$ and $(a'_n)$ of elements of $V$ are {\bf  $\ult$-equal} if $a_n = b_n$ for a.e.\! $n$. It is easy to see that $\ult$-equality is an equivalent relation on $V$. For a sequence $(a_n)$ of elements of $V$, we let $(a_n)/\ult$ denote its equivalence class under $\ult$-equality. The {\bf ultraproduct} $\prod_\ult A_n$ of a sequence $(A_n)$ of subsets of $V$ with respect to $\ult$ is the set 
 $$ \{ (a_n)/\ult : a_n \in A_n \text{ for each } n \}. $$
 The ultraproduct $\prod_\ult A_n$ can be seen as an averaging of $(A_n)$ reflecting certain types of behaviors that hold in $\ult$-a.e.\! $A_n$. The following fact makes this intuition precise.

\begin{fact} \label{fact:ultrafacts} Fix a nonprincipal ultrafilter $\ult$.
Suppose $(A_n)$, $(B_n)$, and $(C_n)$ are a sequence of sets, $\ult$ is a nonprincipal ultrafilter on $\NN$, $A =\prod_\ult A_n$,  $B =\prod_\ult B_n$, and $C =\prod_\ult C_n$. Then 
\begin{enumerate}
    \item  We have $A =\emptyset$ if and only if $A_n =\emptyset$ for a.e.\! $n$. 
    \item We have $A \subseteq B$ if and only if $A_n \subseteq B_n$ for a.e.\! $n$. 
    \item $A\cup B = \prod_{\ult} (A_n \cup B_n)$, $A\cap B = \prod_{\ult} (A_n \cap B_n)$, $A \setminus B = \prod_{\ult} (A_n \setminus B_n)$.
    \item With $\times$ denoting the Cartesian product, $A \times B = \prod_{\ult} (A_n \times B_n)$.

    \item Suppose $\pi: B \times C \to B$ and $\pi_n: B_n \times C_n \to B_n$ for all $n$ are the projections to the first coordinates, and assume $A \subseteq B \times C$ and $A_n \subseteq B_n \times C_n$ for all $n$. Then 
    $$  \pi(A) = \prod_\ult \pi_n(A_n). $$
\end{enumerate}
\end{fact}

The following ``compactness'' property (also called $\aleph_1$-saturated property)  is a new feature of sets obtained by ultraproducts. It is behind a major advantage with working with ultraproduct: One can replace approximation by actual equality if willing to give up some quantitative information.

\begin{fact} \label{fact:saturationofultraproduct}
    Suppose $(A_n)$ is a sequence of sets, each obtained by taking ultraproducts of a sequence of sets in $V$. If for each finite $I \subseteq \NN$, the intersection $\bigcap_{n \in I} A_n$ is nonempty, then the intersection  $\bigcap_n A_n$ is nonempty.
\end{fact}

The following remark, on the other hand, highlight a difficulty of working with ultraproduct, namely, we completely lose meaningful topological data.

\begin{remark} \label{remark: notopology}
Let $A_n =\{ 0,1\} $, and $A=\prod_{\ult}A_n$. Using (the dual of) Fact~\ref{fact:saturationofultraproduct}, it can be shown that $A$ is uncountable.  Suppose we equip $A_n$ with the discrete topology. Then $A$ must be ``pseudo-compact'', and the singleton set $\{a\}$ for each $a\in A$ is ``pseudo-open''. The collection $\{  \{a\} : a\in A\}$ forms a cover of $A$, but there is no finite subcover.
\end{remark}

As we are interested in handling group and products of sets on them, the following fact is needed:

\begin{fact}  \label{fact: Groupunderultraproduct}
    Fix a nonprincipal ultrafilter $\ult$. Suppose $(G_n)$ is a sequence of groups,  $G$ is the ultraproduct $\prod_\ult G_n$ of the underlying set. 
    \begin{enumerate}
        \item Then the map
    $$ \left(\prod G_n\right)^{[2]} \to \left(\prod G_n\right),   ((a_n), (b_n)) \mapsto (a_nb_n) $$
    induces a binary operation from $G\times G$ to $G$ with $((a_n)/\ult, (b_n)/\ult) \mapsto (a_nb_n)/\ult $, and the set $G$ together with this binary operation map is a group.
    \item Suppose $(A_n)$ and $(B_n)$ are sequences with $A_n, B_n\subseteq G_n$, $A =\prod_\ult A_n$, $B =\prod_\ult B_n$, and $AB$ is the product set of $A$ and $B$ with respect to the group operation defined in (1). Then $AB = \prod_\ult A_nB_n$.
    \end{enumerate}
\end{fact}

For a sequence $(G_n)$ of groups and  a nonprincipal ultrafilter $\ult$ on $\NN$, the {\bf  ultraproduct} of $(G_n)$ with respect to $\ult$, denoted by $\prod_\ult G_n$, is the group $G$ whose underlying set is the ultraproduct of the sequence of underlying sets of $(G_n)$,  and whose group operation on $G$ is given by Fact~\ref{fact: Groupunderultraproduct}(1).

A {\bf pseudo-compact group} $G$ is a group equipped with the additional data of a sequence $(G_n)$ of compact groups and a nonprincipal ultrafilter $\ult$ on $G$ such that $G =\prod_\ult G_n$. Very often, we will treat $G$ just as a group and suppress the data about $(G_n)$ and $\ult$. We will say the pseudo-compact group $G$ is built from a sequence $(G_n)$ of local groups and an ultrafilter $\ult$ if we want to make the other pieces of information clear. Note that by Remark~\ref{remark: notopology}, there is no obvious topology one can equip on such $G$. There is also an obvious notion of pseudo-locally-compact groups, but we will not discuss this further as it will not be used.

Unlike topological information, we will see that measure-theoretic information are meaningfully preserved by ultraproducts. Suppose $G$ is a pseudo-locally-compact group built from a sequence $(G_n)$ of groups and a nonprincipal ultrafilter $\ult$. We say that $A \subseteq G$ is {\bf pseudo-measurable} $A$ if there is a sequence $(A_n)$ with $A_n \subseteq G_n$ measurable such that $A = \prod_\ult A_n$. It follows from Fact~\ref{fact:ultrafacts}(4) that the collection of pseudo-measurable subgroups of $G$ forms a Boolean algebra (closed under taking finite intersection, finite union, and relative complement).

A pleasure working with ultrafilter is the following easy fact:

\begin{fact} \label{fact:welldefinedoflimit}
   Fix a nonprincipal ultrafilter $\ult$ on $\NN$. Let $(r_n)$ be a sequence of real numbers. Then exactly one of the following three scenarios hold:
   \begin{enumerate}
       \item There is $r\in \RR$ such that for all $\epsilon>0$, we have $|r_n -r|< \epsilon$ for a.e.\! $n$.
       \item for all $C>0$, we have $r_n >C$ for a.e.\! $n$.
       \item for all $C>0$, we have $r_n <-C$ for a.e.\! $n$.
   \end{enumerate}
\end{fact}

For a sequence $(r_n)$ of real numbers and nonprincipal ultrafilter $\ult$, the {\bf limit of $(r_n)$ under $\ult$}, denoted by $\lim_\ult r_n$, is defined to be the unique element in $\RR \cup \{\pm \infty\}$ such that one of the three cases of Fact~\ref{fact:welldefinedoflimit} holds. This notion of limit behaves as it should. Fact~\ref{fact:behavioroflimit} records some behavior that we will actually use.

\begin{fact} \label{fact:behavioroflimit}
 Fix a nonprincipal ultrafilter $\ult$ on $\NN$.  Suppose  $(r_n)$, $(s_n)$, and $(t_n)$ are sequences of nonnegative real number, and $K$ is a constant. Then we have the following
 \begin{enumerate}
     \item $\lim_\ult (r_n +s_n) =\lim_\ult r_n +\lim_\ult s_n$
     \item  $\lim_\ult Kr_n = K\lim_\ult r_n$
     \item if $r_n \leq s_n$ for a.e.\! $n$, then $\lim_\ult r_n \leq  \lim_\ult s_n$.
     \item If $f: \RR^3 \to \RR$ is a continuous function, and $\lim_\ult r_n, \lim_\ult s_n, \lim_\ult t_n < \infty$, then 
     $$f\left(\lim_\ult r_n, \lim_\ult s_n, \lim_\ult t_n\right) = \lim_\ult f(r_n, s_n, t_n).$$ 
 \end{enumerate}
 Here, we hold the convention that $r+\infty= \infty +\infty = K\cdot \infty =\infty$ and $r < \infty$ for $r \in \RR^{>0}$.
\end{fact}

Using Fact~\ref{fact:behavioroflimit}(1), one can deduce Lemma~\ref{lem:welldefinedofmeasure} below. We omit the obvious proof.
\begin{lemma} \label{lem:welldefinedofmeasure}
    Suppose $G$ is a pseudo-locally-compact group built from a sequence $(G_n)$ of locally compact groups and ultrafilter $\ult$. Let $(\mu_n)$ be a sequence with $\mu_n$ a left Haar measure on $G_n$. Then we have the following:
    \begin{enumerate}
        \item  If $A \subseteq G$ is pseudo-measurable,  $A = \prod_\ult A_n = \prod_\ult A'_n$ for two sequences $(A_n)$ and $(A'_n)$ where $A_n, A'_n \subseteq G_n$ are measurable, then 
        $$ \lim_\ult \mu_n(A_n) = \lim_\ult \mu_n(A'_n).$$
        \item  If $A, B \subseteq G$ are pseudo-measurable and disjoint,  $A = \prod_\ult A_n$ and $B = \prod_\ult B_n$ for two sequences $(A_n)$ and $(B_n)$ where $A_n, B_n \subseteq G_n$ are measurable, then 
        $$ \lim_\ult \mu_n(A_n \cup B_n) = \lim_\ult \mu_n(A_n) + \lim_\ult \mu_n(B_n).$$
    \end{enumerate}

\end{lemma} 
 
Suppose $G$ is a pseudo-compact group built from a sequence $(G_n)$ of compact groups and an ultrafilter $\ult$. A measure $\mu$ on $G$ is a {\bf pseudo-Haar measure} if the following holds:
\begin{enumerate}
    \item[(PH1)] there is a sequence $(\mu_n)$ with $\mu_n$ a Haar measure on $G_n$ such that for every pseudo-measurable $A \subseteq G$, we have
$$  \mu(A)= \lim_{\ult} \mu_n(A_n)$$
where $(A_n)$ is a sequence with $A_n \subseteq G$ measurable and $A =\prod_\ult A_n$
    \item[(PH2)] $\mu$ is the completion of its restriction to the $\sigma$-algebra generated by pseudo-measurable sets.
\end{enumerate}
To simplify notation, we will write (PH1) as $\mu = \lim_\ult \mu_n$ on pseudo-measurable sets.  By Fact~\ref{fact:behavioroflimit}(2), a constant multiple of a pseudo-Haar measure is a pseudo-Haar measure. We note that every pseudo-measurable set is a measurable set  with respect to a pseudo-Haar measure, but the converse is not true. The next lemma shows that a pseudo-Haar measure can be constructed a sequence of Haar measure.

\begin{lemma} \label{lem: ExistenceofpseudoHaar}
    Suppose $G$ is a pseudo-compact group built from a sequence $(G_n)$ of compact groups and an ultrafilter $\ult$. Let $(\mu_n)$ be a sequence where $\mu_n$ is a Haar measure on $G_n$. Then there is a pseudo-Haar measure on $G$ such that $\mu = \lim_\ult \mu_n$ on pseudo-measurable subsets of $G$.
\end{lemma}

\begin{proof}
    Let $\mu_0$ be the function on the Boolean algebra of pseudo-measurable subsets of $G$ given by
     $$  \mu_0(A)= \lim_{\ult} \mu_n(A_n)$$
where $(A_n)$ is a sequence with $A_n \subseteq G$ measurable and $A =\prod_\ult A_n$. We note that $\mu_0$ is well-defined by  Fact~\ref{lem:welldefinedofmeasure}(1). From the definition of limit and ultraproduct, we have $\mu_0(\emptyset)=0$. If a pseudo-measurable $A\subseteq G$ is a countable union of pseudo-measurable subsets of $G$, it follows from Fact~\ref{fact:saturationofultraproduct} that $A$ is equal to a finite union of these subsets of $G$. Hence, it follows from  Fact~\ref{lem:welldefinedofmeasure}(2) that $\mu_0$ is a premeasure. Thus, by Carath\'eodory's extension theorem (Fact~\ref{Caratheodory}) $\mu_0$ can be extended to a measure on the $\sigma$ algebra of subsets of $G$ generated by the pseudo-measurable sets. Using completion, we obtain a pseudo-Haar measure with the desired properties.
\end{proof}

The next lemma shows that a pseudo-compact group is in a sense unimodular:

\begin{lemma} \label{invarianceofpseudoHaar}
    Suppose $G$ is pseudo-compact,  $A \subseteq G$ is pseudo-measurable, and $\mu$ is a pseudo-Haar measure on $G$. Then for all $g \in G$, the sets $gA$ and $Ag$ are pseudo-measurable, and $\mu(gA)=\mu(A) =\mu(Ag)$.
\end{lemma}
\begin{proof}
    We will only show $gA$ is pseudo-measurable and $\mu(gA) =\mu(A)$, as the remaining parts are similar. Suppose $G$ is built from the sequence of compact subgroup $G_n$ and a nonprincipal ultrafilter $\ult$. Let $(g_n)$, $(A_n)$, and $(\mu_n)$ be sequences such that $g_n \in G_n$,  $g = (g_n)/\ult$, $A_n \in A$ is measurable, $A = \prod_\ult A_n$, $\mu_n$ is a Haar measure on $G_n$,  and $\mu(A) = \lim_\ult \mu_n(A_n)$. One can then check that $gA = \prod_\ult (g_n A_n)$, so $gA$ is pseudo-measurable. The equality $\mu(gA)=\mu(A)$ follows from the equality $\mu_n(A_n) =\mu_n(g_nA_n)$.
\end{proof}

Let $G$ be a pseudo-compact group. We say that pseudo-measurable sets $A, B \subseteq G$ are {\bf commensurable} if there is a pseudo-Haar-measure $\mu$ on $G$ such that 
$$ 0< \mu(A), \mu(B) < \infty. $$
We say that $A$ is {\bf infinitesimal} compared to $B$ if there is a pseudo-Haar-measure $\mu$ on $G$ such that 
$$ 0 < \mu(A) < \infty \text{ and } \mu(B) = \infty. $$

The following lemma is the counterpart of the uniqueness up to constant of Haar measures.

\begin{lemma} \label{fact:almostuniqueness}
    Let $G$ be a pseudo-compact group, and $\mu$ and $\mu'$ be pseudo-Haar measures on $G$. Suppose there is a pseudo-measurable set $A \subseteq G$ such that $0< \mu(A) < \infty$ and $0 <  \mu'(A) < \infty$.  Then there is  $K \in \RR^{\geq 0}$ such that $\mu'= K \mu$.
\end{lemma}

\begin{proof}
Suppose $G$ is built from the sequence $(G_n)$ and the ultrafilter $\ult$. 
Let $K = \mu'(A)/\mu(A)$. We will show $\mu'(B) = K\mu(B)$ for an arbitrary pseudo-measurable $B\subseteq G$. This is enough because $\mu$ can be obtained uniquely via completion and Carath\'eodory's theorem from the premeasure $\mu_0$ which is the restriction of $\mu$ to the pseudo-measurable subsets of $B$. Let $(A_n)$ and $(B_n)$ be  sequences such that $A_n, B_n \subseteq G_n$ are measurable,  $A = \prod_\ult A_n$, and $B = \prod_\ult B_n$. Let $(\mu_n)$ and $(\mu'_n)$ be two sequences such that $\mu_n$ and $\mu'_n$ are Haar measure on $G_n$, and 
\[
\mu(A) = \lim_\ult \mu_n(A_n) \quad\text{and}\quad  \mu'(A) = \lim_\ult \mu'_n(A_n).
\]
By definition, for any $\epsilon>0$, for a.e.\! $n$, we have 
\[
\mu_n(A_n) < (1+\epsilon) \mu(A) \quad\text{and}\quad \mu'(A)< (1+\epsilon)\mu_n(A_n).
\]
Hence, for these $n$, $K\mu_n(A_n)< (1+\epsilon)^2 \mu'_n(A_n)$. Since, $\mu_n$ and $\mu'_n$ differs by a constant, it leads to $K \mu_n(B_n) < (1+\epsilon)^2 \mu'_n(B_n)$. It follows that $\mu(B)\leq  K(1+\epsilon)^2\mu'(B)$. Since $\epsilon$ can be taken arbitrarily, we get $\mu'(B) \leq K (B)$. A similar argument yields, $\mu'(B) \geq K \mu $. Thus, $\mu'(B) =K \mu(B)$ completing the proof.
\end{proof}

We deduce some consequence for commensurability and being infinitesimal. There are other facts along this line, but we leave those to interested readers.

\begin{corollary} \label{cor: commensurability}
Suppose $G$ is a pseudo-compact group, and $A, B, C \subseteq G$ are pseudo-measurable. Then we have the following:
\begin{enumerate}
    \item If $A$ and $B$ are commensurable, and $B$ and $C$ are commensurable, then $A$ and $C$ are commensurable.
    \item If $A$ and $B$ are commensurable, and $A$ is infinitesimal compared to $C$, then $B$ is infinitesimal compared to $C$.
\end{enumerate}
\end{corollary}
\begin{proof}
    We will only prove (1), as the proof of (2) is similar. Let $\mu$ and $\mu'$ be pseudo-Haar measure on $G$ such that $0< \mu(A), \mu(B)< \infty$ and $0< \mu'(B), \mu'(C)<0$. Then by Lemma~\ref{fact:almostuniqueness}, $\mu$ and $\mu'$ differs by a constant $K$. It follows that $0< \mu(A)< \infty$, which yields the desired conclusion.
\end{proof}

Let $G$ be a pseudo-compact group. Suppose $A,B \subseteq G$ are such that $A$, $B$, and $AB$  are pseudo-measurable, and $A$ and $B$ are commensurable. Let $\mu$ be a pseudo-Haar measure on $G$ such that $0< \mu(A), \mu(B)< \infty$. Note that $0< \mu(AB)$ by Fact~\ref{invarianceofpseudoHaar}.
We define the {\bf Brunn--Minkowski growth } $\BM(A,B)$ to be the unique $r >0$ such that 
$$ \ \mu(AB)^{\frac{1}{r}} = \mu(A)^{\frac{1}{r}} +\mu(B)^{\frac{1}{r}}$$
if $\mu(AB)< \infty$, otherwise we set $\BM(A,B)=\infty$.
 By Fact~\ref{fact:almostuniqueness}, this definition is independent of the choice of the pseudo-Haar measure $\mu$ as long as $0< \mu(A)< \mu(B)< \infty$. 

We now prove the main result of this section:

\begin{proposition} \label{prop: smallexpansionultraproduct}
     Suppose $G$ is a pseudo-compact group built from a sequence $(G_n)$ of compact groups and a nonprincipal ultrafilter $\ult$. Let $(A_n)$, $(B_n)$ and $(\mu_n)$ be sequences such that $A_n, B_n, A_nB_n \subseteq G_n$ has positive measure, $\mu_n$ is a Haar measure on $G_n$. Let $A, B \subseteq G$ be pseudo-measurable sets and $\mu$ a pseudo-Haar measure on $G$ such that $A =\prod_\ult A_n$ and $B=\prod_\ult B_n$. Then we have the following.
\begin{enumerate}
    \item If there is a constant $N$ such that $\mu_n(A_n)/N \leq \mu_n(B_n) \leq N\mu_n(A)_n$,  then $A$ and $B$ are commensurable. 
     \item If $\lim_{n \to \infty} \mu_n(A_n)/\mu_n(G_n) = \lim_{n \to \infty} \mu_n(B_n)/\mu_n(G_n)=0$, then $A$ and $B$ are infinitesimal compared to $G$.
    \item If $A$ and $B$ are commensurable and  $r$ is a constant such that  $\BM(A_n,B_n) \leq r$, then $\BM(A,B) \leq r$.
   \end{enumerate}
\end{proposition}

\begin{proof}
We first prove (1). For each $n$, scaling  $\title{\mu}_n$ by a constant factor, we can assume that $\mu_n(A_n)=1$. Let $\mu$ be a pseudo-Haar measure on $G$ such that  $\mu = \lim_\ult \mu_n$. We can check that $\mu(A)=1$ and $1/N<\mu(B)< N$. This shows $A$ and $B$ are commensurable.

The proof for (2) can be obtained similarly. We now prove (3). Let $\mu$ be a pseudo-Haar measure on $G$ such that $0< \mu(A), \mu(B)< \infty $. By suitably scaling,  assume $\mu = \lim_\ult \mu_n$. Then we have 
\[
\lim_\ult \mu_n(A_n) = \mu(A),\, \lim_\ult \mu_n(B_n) =\mu(B), \text{ and }\lim_\ult \mu_n(A_nB_n) =\mu(AB).
\]
The condition that $\BM(A_n, B_n) \leq r$ implies $\mu_n(A_n B_n) \leq (\mu_n(A_n)^{\frac1r} +\mu_n(B_n)^{\frac1r})^r$. It follows from Fact~\ref{fact:behavioroflimit}(3,4) that we will also get $\mu(AB) \leq (\mu(A)^{\frac1r}+\mu(B)^{\frac1r})^r$. In particular, this implies $\mu(AB)<\infty$. Apply Fact~\ref{fact:behavioroflimit}(4) again, we get  $\BM(A,B) = \lim_\ult \BM(A_n, B_n)$. The conclusion then follows from Fact~\ref{fact:behavioroflimit}(3). 
\end{proof}

\section{Approximate groups from small growth}

In this section, we will link the sets in the ultraproduct having small measure product with approximate groups, the intermediate object used to connect to noncompact Lie groups.

For a constant $K$, we call $S \subseteq G$ a {\bf $K$-approximate group} if 
\begin{enumerate}
\item $\id_G \in S$,
\item $S= S^{-1}$,
\item $S^2$ can be covered by $K$-many left translates of $S$ (equivalently, $K$-many right translates of $S$).
\end{enumerate} 
Clearly, there is no $K$-approximate group if $K<1$, and a $1$-approximate group is a subgroup. We say that $S \subseteq G$ is an {\bf approximate group} if it is a $K$-approximate group for some $K$. For instance, a open subset of a compact group is an approximate group. 

 If $G$ is a compact and connected group and $S \subseteq G$ is a $2^r$-approximate group such that $S^2$ is also measurable, then $\BM(S,S) \leq r$. Hence, under rather mild assumption that $S^2$ is still measurable, an approximate group has small measure doubling.

The following fact from~\cite{T08} establish a partial converse: Approximate groups arise from sets with small measure growth.

\begin{fact}  \label{fact: FromTaopaper}
Let $G$ be a compact group equipped with a Haar measure $\mu$, and $A, B \subseteq G$ are such that $A$, $B$, and $AB$ are measurable. Suppose $\mu(A)/N < \mu(B) < N\mu(A)$, and $\BM(A,B)\leq r$. Then there is an open $O_{N,r}(1)$-approximate group $S$ such that $\mu(S) \sim_{N,r} \mu(A)$, $A$ is contained in $O_{N, r}(1)$ left translates of $S$, and $B$ is contained in $O_{N, r}(1)$ right translates of $S$.
\end{fact}

The following lemma provides the first sign that approximate groups are robust and easier to handle.

 \begin{lemma} \label{lem: interativesmallgrowthapproximategroups}
     Suppose $S \subseteq G$ is a $K$-approximate group. Then $S^n$ can be covered by $K^{n-1}$ left translates of $S$ when $n >0$.
 \end{lemma}

 \begin{proof}
For $n=1$, the statement follows from the definition. Suppose we have shown the conclusion for $n$, then $S^{n+1} = S^n S$ can be covered by $K^{(n-1)}$ left translates of $S^2$, which can in turn be covered by $K^n$ left translates of $S$.
 \end{proof}

In order to be able to use the model-theoretic machinery later on,  we need special kind of approximate groups.  Let $\Omega$ be an ambient set. A {\bf structure} $\Sigma$ on $\Omega$ is a sequence $(\Sigma_n)$ satisfying the following conditions:
\begin{enumerate}
    \item [(S1)]  For each $n$, $\Sigma_n$ is a Boolean subalgebra of $\cP( \Omega^{[n]})$.
    \item [(S2)] The diagonal $\{(a,a) : a\in \Omega\}$ is in $\Sigma_2$.
    \item [(S3)] If $m \leq n$ and $\pi: \Omega^{[n]} \to \Omega^{[m]}$ is the projection to some $m$ out of  $n$ coordinates, $A \subseteq \Omega^{[m]}$ is in $\cD_m$, then $\pi^{-1}(A)$ is in $\cD_n$.
    \item[(S4)] If $m \leq n$ and $\pi: \Omega^{[n]} \to \Omega^{[m]}$ is the projection to some $m$ out of  $n$ coordinates, $A \subseteq \Omega^{[n]}$ is in $\cD_n$, then $\pi(A)$ is in $\cD_m$.
\end{enumerate}

We say that $D \subseteq \Omega^{[n]}$ is {\bf definable} in $\Sigma$ if $D$ is an element of $\Sigma_n$. A subset of $\Omega^{[n]}$ is {\bf $\delta$-definable} in $\Sigma$ if it is a countable intersection of subsets of $\Omega^{[n]}$ definable in $\Sigma$, and  a subset of $G^{[n]}$ is {\bf $\sigma$-definable} in $\Sigma$  if it is a countable union of subsets of $G^{[n]}$ definable in $\Sigma$.

A structure $\cD$ on an ambient set $\Omega$ is {\bf $\aleph_1$-saturated} if whenever $(D_n)$ is a sequence of subsets of $\Omega^k$ is definable in $\cD$, and every finite intersection is nonempty, then $(D_n)$ has nonempty intersection.
The following easy observation about $\aleph_1$-saturation will become very useful later on.

\begin{lemma} \label{lem: Interpolation}
    Suppose $\Sigma$ is an $\aleph_1$-saturated structure on an ambient set $\Omega$. If $A \subseteq \Omega^{[n]}$ is $\delta$-definable in $\Sigma$, $B \subseteq  \Omega^{[n]}$  is $\sigma$-definable in $\Sigma$, and $A \subseteq B$, then there is $D \subseteq \Omega^{[n]}$ definable in $\Sigma$ such that $A \subseteq D \subseteq B$.
\end{lemma}

\begin{proof}
    Suppose $A = \bigcap_m A_m$ and $B = \bigcup_m B_m$  with $A_m, B_m \subseteq \Omega^{[n]}$ definable in $\Sigma$ for each $m$. Then 
    $$ \left( \bigcap_m A_m\right) \cap \left( \bigcap_m (\Omega^{[n]  }\setminus B_m\right) = \emptyset. $$
    By $\aleph_1$-saturation, there is $N$ such that $ \left( \bigcap_{m=1}^N A_m\right) \cap \left( \bigcap_{m=1}^N (\Omega^{[n]  }\setminus B_m\right) = \emptyset. $ it is easy to check that $D = \left( \bigcap_{m=1}^N A_m\right)$ satisfies desired conditions.
\end{proof}

An {\bf expansion} of a group is a pair $(G, \Sigma)$ such that $G$ is a group,  $\Sigma$  is a structure  on the underlying set of $G$, and the graph $\Gamma= \{ (a,b,c) \in G^3: ab=c\}$  of  multiplication in $G$ is definable in $\Sigma$. 
The following lemma will be useful later on.

\begin{lemma} \label{Lemma: productaredefinable}
Let $(G, \Sigma)$ be an expansion of a group. If $A,B \subseteq G$ are definable in $\Sigma$, then $AB \subseteq G$ is definable in $\Sigma$. Hence, if $A \subseteq G$ is definable in $\Sigma$, then $A^n \subseteq G$ is definable in $\Sigma$ for all $n$.    
\end{lemma}

\begin{proof}
    Let $A,B$ be as in the statement of the lemma. By (S3) $A \times G \times G \subseteq G^{[3]}$ and $G \times B \times G \subseteq G^{[3]}$ are definable in $\Sigma$. Set 
    $$C :=(A \times G \times G) \cap (G \times B \times G) \cap \Gamma.$$
    Then $C \subseteq G^{[3]}$ is definable in $\Sigma$ by (S3) and the definition of an expansion.
  Note that $AB \subseteq G$ is the projection $C\subseteq G^{[3]}$ to the last coordinate, hence $AB$ is definable in $\Sigma$ by (S4). The second statement of the lemma is immediate from the first.
\end{proof}

For a collection $(A_i)_{i \in I}$
of subsets of $G$ there is clearly a smallest expansion $(G, \Sigma)$ of the group $G$ where, for each $i \in I$, the set $A_i$ is definable in $\Sigma$. We call this $(G, \Sigma)$ the expansion {\bf generated by} $(A_i)_{i \in I}$. An expansion $(G, \Sigma)$ of a group is {\bf $\aleph_1$-saturated} if the structure $\Sigma$ is $\aleph_1$-saturated. This property appears naturally here due to its relationship with ultra-product construction.

\begin{lemma} \label{lem: Saturatedofgeneratedexpansion}
    Suppose $(G, \Sigma)$ is an expansion of a group generated by a collection of pseudo-measurable subsets of $G$. Then $\Sigma$ is $\aleph_1$-saturated.
\end{lemma}
\begin{proof}
Suppose $G$ is built from the sequence $(G_n)$ of compact groups and a nonprincipal ultraproduct $\ult$.
For each $k$, let $\Sigma'_k$ be the collection of subsets of $G^{[k]}$ of the form $C = \prod_\ult C_n$    with $C_n \subseteq (G_n)^{[k]}$. It follows from Fact~\ref{fact:ultrafacts}, that $\Sigma': =(\Sigma'_k)$ is a structure. Note that the graph $\Gamma$ is definable in $\Sigma'$, and by definition $A$, $B$, and $S$ are also definable in $\Sigma'$ as they are pseudo-measurable. Thus, we have $\Sigma_k \subseteq \Sigma'_k$ for all $k$. By Fact~\ref{fact:saturationofultraproduct}, $\Sigma'$ is $\aleph_1$-saturated, which implies $\Sigma$ is $\aleph_1$-saturated.
\end{proof}

We now arrive at the main concept which will allow us to use the model-theoretic machinery later on.
Suppose $G$ be a group, and $(G, \Sigma)$ is an expansion of $G$. We say that $S$ is {\bf definably amenable} in $(G, \Sigma)$ if there is a finitely-additive left-invariant measure $\mu$ on $\langle S \rangle$ satisfying:
\begin{enumerate}
    \item[(DA1)] $\mu(S)=1$
    \item[(DA2)] Every $D \subseteq \langle S \rangle$ definable in $\Sigma$ is $\mu$-measurable.
\end{enumerate}
In view of Lemma~\ref{Lemma: productaredefinable}, one can see this as a strengthening of the condition that $S^n$ is measurable for each $n >0$. 
The next lemma explain why this can arise in our setting.

\begin{lemma} \label{lem: Definablyamenablearise}
Suppose $(G, \Sigma)$ is an expansion of a pseudo-compact group such that every $D \subseteq G$ definable in $\Sigma$ is pseudo-measurable, and there is a pseudo Haar measure $\mu$ such that $0< \mu(S)< \infty$. 
 Then $S$ is definably amenable in $(G, \Sigma)$.
\end{lemma}

\begin{proof}
Recall that every pseudo measurable set is $\mu$-measurable.  Define the finitely additive measure $\nu$ on $\langle S \rangle$ by setting $\Tilde{\mu}(D) = \mu(D)/\mu(S)$ for $D \subseteq \langle S \rangle$ definable in $\Sigma$.  The desired conclusions are immediate.
\end{proof}

We now discuss how to obtain the conditions for Lemma~\ref{lem: Definablyamenablearise} in our setting. Recall that $X$ is {\bf semialgebraic} if it is a finite union of the solution sets of systems of algebraic inequalities. 
The following fact is a restatement of the Tarski-Seidenberg theorem~\cite[(2.10)]{Loubook}.

\begin{fact} \label{Fact: Tarski-Seidenberg}
Let  $\cD_n$ be the collection of semialgebraic subsets of $\RR^n$. Then $\Sigma: =(\Sigma_n)$ is a structure on $\RR$.
\end{fact}

We say $X \subseteq \RR^n$ is {\bf algebraic} if it is the solution set of a system of polynomial equation with coefficient in $\RR$. Identifying the underlying set of the general linear group $\Gld$ in the obvious way as an algebraic subset of $\RR^{d^2}$, we say that $G \leq \Gld$ is an {\bf algebraic subgroup} of $\Gld$ if its underlying set is algebraic. 
A subset of $G^{[n]}$ is {\bf semi-algebraic} if it is semialgebraic under the obvious identification of $G^{[n]}$ with a subset of $\RR^{nd^2}$. The Tarski-Seidenberg theorem translates into the following lemma for $G$:

\begin{lemma} \label{Lemma: Tarski-Seidenberg2}
Let $G$ is an algebraic subgroup of $\Gld$, and $\cD_n$ the collection of semialgebraic subsets of $G^{[n]}$. Then $(G, \Sigma)$ is an expansion of the group $G$ where we set $\Sigma: =(\Sigma_n)$.
\end{lemma}

\begin{proof}
Note that the diagonal $\{(g, g): g \in G\}$ can be obtained as the intersection of $G \times G$ and an $n$-times Cartersian product of the diagonal $\{(x, x): x \in \RR \}$. Also note that the projection $G^{[n]} \to G^{[m]}$ to $m$ out of $n$ coordinates can be obtained from a suitable projection $\RR^{nd^2} \to \RR^{md^2}$. Finally, the graph of multiplication in $G$ is semialgebraic, in fact, algebraic.
The conclusion then follows easily from Fact~\ref{Fact: Tarski-Seidenberg}.
\end{proof}

We next lemma allow us to replace small-growth pairs $(A, B)$ with semialgebraic ones.

\begin{lemma} \label{lem: semialgebraicaproximate0}
    Suppose $G$ is an algebraic subgroup of $\Gld$, $\mu$ is a Haar measure of $G$, and $A, B \subseteq G$ are such that $A$,  $B$, and $AB$ are measurable,  $\mu(A)/N < \mu(B) < N \mu(A)$, and $\BM(A,B) \leq r$. Then for all $\epsilon>0$, there are semialgebraic $A', B' \subseteq G$ such that $\mu(A)/(1+\epsilon) < \mu(A') < (1+\epsilon) \mu(A)$,  $\mu(B)/(1+\epsilon) < \mu(B') < (1+\epsilon) \mu(B)$, and $\BM(A',B') \leq r+\epsilon$.
\end{lemma}

\begin{proof}
   We will deduce the lemma from two weaker statements.\medskip

   \noindent {\bf Claim 1.} The statement of the lemma holds if we replace the semialgebraic requirement on $A'$ and $B'$ with the requirement that they are compact.

\medskip

   \noindent{\it Proof of claim 1.}  Recall that $\mu$ is the completion of its restriction to Borel sets and $\mu$ has outer regularity. Applying these properties to $G\setminus A$ and $G\setminus B$, for any $\delta$ can we obtain compact $A' \subseteq A$ and $B' \subseteq B$ such that $\mu(A') > \mu(A)/(1+\delta)$ and $\mu(B') > \mu(A)/(1+\delta)$. Choosing $\delta$ sufficiently small, we see that the desired properties are satisfied.
   \hfill $\bowtie$
\medskip

    \noindent {\bf Claim 2.} The statement of the lemma holds if we add the assumption that $A$ and $B$ are compact.

\medskip

   \noindent{\it Proof of claim 2.} By a similar argument as in the proof of Claim $1$, we choose open $U \subseteq G$ containing $AB$ such that 
   $\mu(U)< (1+\delta) \mu(AB)$
   where we will determine $\delta$ later.
   Let $d$ be an invariant distance on $G$. Set $ d_0 = d(AB, G \setminus U)$. Let $U_A= \{ x \in G: d(x, A)  < d_0/ 2\}$ and  $U_B= \{ x \in G: d(x, A)  < d_0/2\}$. Then $U_AU_B \subseteq U$. As $A$ is compact, we can choose $A'$ which is a union of finitely many Euclidean open balls in $\Gld$ intersecting $G$ such that $A \subseteq A' \subseteq U_A$. Choose $B'$ similarly such that $B \subseteq B' \subseteq U_B$. With $\delta$ small enough, we see that the conditions are satisfied.
   \hfill $\bowtie$\medskip

Apply Claim 1 and then Claim 2 suitably, we get the statement of the Lemma.
\end{proof}

The next lemma allow us to replace approximate groups with semialgebraic ones.

\begin{lemma} \label{lem: semialgebraicaproximate}
    Suppose $G$ is an algebraic subgroup of $\Gld$, and $S \subseteq G$ is an open $K$-approximate group. Then there is a semialgebraic open $K^3$-approximate group $S'$ such that  $S \subseteq S'\subseteq S^2$. The set $S'$ is definably amenable with respect to some left Haar measure of $G$.
\end{lemma}

\begin{proof}
    Let $\overline{S}$ be the closure of $S$ with respect to the topology in $G$. We first note that $\overline{S} \subseteq S^2$. Indeed, the open neighborhood $aS^{-1}$ of $a \in \overline{S}$ contains a point in $a' \in S$, so $a \in a'S \subseteq S^2$. 
    
    Choose an open cover $(U_i)_{i \in I}$ of $\overline{S}$ such that for each $i$ each $U_i \subseteq S^2$,  and $U_i$ is of the form $B_\epsilon \cap G$ with $B_\epsilon$ a ball of radius $\epsilon$ in $\RR^{d^2}$. Since $G$ is compact, $\overline{S}$ is also compact. Choose a finite subcover of  $(U_i)_{i \in I}$  and take $S'$ to be their union. It is clear that $S \subseteq S' \subseteq S^2$, and $S'$ is semialgebraic. Replace $S'$ by $S' \cup (S')^{-1}$, we can make $S'$ symmetric.

    Finally, note that $(S')^2 \subseteq S^4$, which can be covered by $K^3$ left-translates of $S$, and in turn covered by $K^3$ left-translates of $S'$.   

    We note that $S'$ is an open subset of $G$, so we can choose a left Haar measure $\mu$ of $G$ such that $\mu(S')=1$. By Lemma~\ref{Lemma: Tarski-Seidenberg2}, every subset of $G$ definable in $(G,S)$ is semi-algebraic. In particular, they are Borel, and hence $\mu$-measurable.
\end{proof}

To link to compact Lie groups, we also need the following fact~\cite[Theorem 12.3.9]{hilgert}.

\begin{fact} \label{Fact: compactLierealalgebraic}
    Every compact Lie group is isomorphic as a topological group to an algebraic subgroup of $\Gld$ for some $d$.
\end{fact}

Let $G$ be a pseudo-compact group built from a sequence $(G_n)$ of compact groups and a nonprincipal ultrafiter $\ult$.
We say that $G$ is {\bf pseudo-Lie} if each $G_n$ is also a Lie group.
We now prove the main proposition of this section.

\begin{proposition} \label{Prop: Approximategroupsfromsmallgrowth}
    Suppose $G$ is a pseudo-Lie pseudo-compact group, and $A, B \subseteq G$ are such that $A$,$B$, and $AB$ are pseudo-measurable, $A$ and  $B$ are commensurable, and $\BM(A,B) \leq r$. Then for each $\epsilon>0$, there are $A', B', S \subseteq G$ such that the following conditions hold:
    \begin{enumerate}
        \item Let $(G, \Sigma)$ be the expansion of the group $G$ generated by $A'$, $B'$, and $S$. Then every $D \subseteq G$ definable in $\Sigma$ is pseudo-measurable.  It follows that $S$ is definably amenable in $(G, \Sigma)$.
        \item $A, B, A', B', S$ are all commensurable to one another
        \item $BM(A', B') \leq r+\epsilon$
        \item $S$ is an approximate groups
        \item  $A'$ can be covered by finitely many left-translates of $S$, and $B'$ can be covered by finitely many right-translates of $S$
    \end{enumerate}
\end{proposition}

\begin{proof}
Suppose $G$ is built from a sequence $(G_n)$ of compact Lie groups and a nonprincipal ultrafilter $\ult$. Using Fact~\ref{Fact: compactLierealalgebraic}, we can arrange that for each $n$, $G_n$ is a closed subgroup of a general linear group. Let $(A_n)$, $(B_n)$, and $(C_n)$ be sequences such that $A_n, B_n, C_n \subseteq G_n$ are measurable such that $A =\prod_\ult A_n$, $B= \prod_\ult B_n$, and $AB =\prod_\ult C_n$. Recall that $AB=\prod_\ult(A_nB_n)$ by Fact~\ref{fact: Groupunderultraproduct}(2). Hence, by fact~\ref{fact:ultrafacts}(1,3) we must have $A_nB_n = C_n$ for a.e.\! $n$. In particular, for a.e.\! $n$, the product $A_nB_n$ is measurable.

Now let $\mu$ be a pseudo-Haar measure such that $0< \mu(A), \mu(B)< \infty$. Then there is $N$ such that $\mu(A)/N < \mu(B) < N \mu(A)$.
Let $(\mu_n)$ be a sequence such that $\mu_n$ is a Haar measure on $G_n$ and $\mu =\lim_{\ult} \mu_n$. Using Fact~\ref{fact:behavioroflimit}, for every $\epsilon>0$ and for a.e.\! $n$, 
$$\mu_n(A_n)/(2N)< \mu_n(B_n)< (2N) \mu_n(A_n) \text{ and } \BM(A_n, B_n) \leq (r+\epsilon/2).$$

Using Lemma~\ref{lem: semialgebraicaproximate0}, for a.e. $n$, we obtain semialgebraic $A'_n$ and $B'_n$ such that 
$$\mu(A_n)/2 < \mu(A'_n) < 2 \mu(A_n), mu(B_n)/2 < \mu(B'_n) < 2 \mu(B_n), \text{ and }
 \BM(A'_n, B'_n) \leq (r+\epsilon).$$  Using Fact~\ref{fact: FromTaopaper} together with Lemma~\ref{lem: semialgebraicaproximate}, 
we produce for a.e.\! $n$ an $O_{N,r}(1)$-approximate group $S_n$ such that 
$S_n$ is open, $S_n$ is definably amenable with respect to some Haar measure of $G_n$, we have $\mu_n(S_n) \sim_{N,r} \mu_n(A'_n)$, the set $A'_n$ can be covered by $O_{N,r}(1)$-left translates of $S_n$, and $B'_n$ can be covered by $O_{N, r}(1)$-right translates of $S_n$. 

Let $A' = \prod_\ult A'_n $, $B' = \prod_\ult B'_n $, and   $S = \prod_\ult S_n$. It remains to verify that the conditions of the lemma are satisfied. We note that (3) is immediate from the construction. The set $S$ is an $O_{N,r}(1)$-approximate group, and in particular, an approximate group. Also, $\mu(S) = O_{N,r}(1) \mu(A)$, we see that $S$ is commensurable with $A'$. By Lemma~\ref{cor: commensurability}, $S$ is also commensurable with $A$, $B$, and $B'$. Hence, we also get (2) and (4).

We now verify (1). By scaling $\mu_n$ suitably, we can harmlessly arrange that $\mu_n(S_n)=1$. Then $\mu(S)=1$.  Let $\Sigma'_m$ be the subcollection of $\cP( G^{[m]})$ consisting of $D \subseteq G^{[m]}$ of the form
$$ D= \prod_\ult D_n  $$
where $D_n \subseteq  (G_n)^{[m]} $ is semialgebraic. It follows from Fact~\ref{fact:ultrafacts}, Fact~\ref{fact: Groupunderultraproduct}, and Lemma~\ref{Lemma: Tarski-Seidenberg2} that  with $\Sigma'= (\cD'_m)$ the pair $(G, \Sigma')$ is an expansion of the group $G$. 
As $A'$, $B'$, and $S$ are definable in $\Sigma'$, we have $\Sigma$ is a substructure of $\Sigma'$. In particular, if $D \subseteq G$ is definable in $\Sigma$, then $D$ is also definable in $\Sigma'$. Hence, such $D$ pseudo-measurable, and so $\mu$-measurable. Thus, $S$ is definably amenable by Lemma~\ref{lem: Definablyamenablearise}.

Finally, we verify (5). We can choose $m= m(N,r)$ such that for a.e.\! $n$, there are at most $m$-many left-translates of $S_n$ needed to cover $A'_n$. For these $n$,  adding more left-translates of $S_n$ if necessary, we assume that the $m$ translates $a_{n,1}S_n, \ldots, a_{n, m} S_n$ can be used to cover $A'_n$. For $i \in [m]$, set  $a_i = (a_{n,i})/\ult$. We claim that $(a_i S)_{i \in [m]}$ is a cover of $A'$. This is the case because, 
\[
A'\setminus (a_1S \cup a_mS) = \prod_\ult A'_n\setminus (a_{n,1}S \cup a_{n,m}S)  = \emptyset.
\]
A similar argument shows that finitely many right translates of $S$ is needed to cover $B'$.
\end{proof}

\section{Lie models from approximate groups}

In this section, we will define locally compact models of approximate groups and describe the Hrushovski's locally compact model theorem~\cite{Hrushovski}, which allows us to construct them from  definably amenable approximate groups. In fact, an improvement of Hrushovski's theorem  by Massicot and Wagner~\cite{MW} for the definably amenable case will be discussed. This suits our purpose better even though the original theorem  by Hrushovski also suffices with some extra steps. We also supplement this result with an additional step which allows us to build connected Lie models. This later step is folkloric, so we only gather the facts together. (Note that  Hrushovski's locally compact model theorem is also called Hrushovski's Lie model theorem. This is in viewed of the fact that this theorem can be used in combination with the Gleason--Yamabe theorem to build Lie models. We split them apart here as it is done so in~\cite{MW}, and we also think it is conceptually clearer.)

Let $G$ be a group, and $S \subseteq G$ is an approximate subgroup. A {\bf locally compact model} of $S$ is a surjective group homomorphism such that the following two conditions are satisfied: 
\begin{enumerate}
    \item[(LM1)] (Thick image) There is an open neighborhood $U$ of $\id_L$ such that $\pi^{-1}(U) \subseteq S$ (consequently, $\ker \pi \subseteq S$ and $U \subseteq \pi(S)$).
    \item[(LM2)] (Compact image) $\pi(S)$ in $L$ is precompact.
\end{enumerate}
Behind this definition is the observation that approximate groups arise naturally from group homomorphisms into locally compact groups. This can be seen through the lemma below:
\begin{lemma} \label{lem: motivationforlocallycompactmodel}
    Suppose $G$ is a group, $S$ is a subset of $G$  with  $\id_G \in S$ and $S =S^{-1}$,  $L$ is a locally compact group,  and $\pi: \langle S \rangle \to L$ is such that conditions (LM1) and (LM2) are satisfied,
          Then $S$ is an approximate group. In particular, an open an precompact subset of a locally compact group is an approximate group.
\end{lemma}

\begin{proof}
    Indeed, by (LM2), such $S^2$ will be contained in an inverse image $\pi^{-1}(F)$ of a compact set $F \subseteq L$, so $S^2$ is covered by finitely many left translates of $\pi^{-1}(U) \subseteq S$ in (LM1). The last sentence is a special case with $G=L$ and $\pi$ the identity map.
\end{proof}

We caution the reader that the word ``model'' in ``locally compact model'' is unrelated to model-theory. This word used comes from the Freiman homomorphism and Ruzsa's model lemma in the abelian settings. 

Let  $G$ be a group,  $S \subseteq G$ is an approximate group, and $\pi: \langle S  \rangle \to L$ is a locally compact model of $S$. If $L$ is a Lie group, we call $\pi$ a {\bf Lie  model} of $S$. We say that $\pi: \langle S  \rangle \to L$ is {\bf noncompact} if $L$ is noncompact. We define {\bf unimodular} and {\bf connected} for $\pi$ similarly.

Let $(G, \Sigma)$ be an $\aleph_1$-saturated expansion of a group $G$, and $H \subseteq G$ is $\sigma$-definable in $\Sigma$. A surjective group homomorphism  $\pi: \langle H \rangle \to L$ with $L$ a locally compact group is {\bf definable} (in the sense of continuous logic) in $\Sigma$ if the topology on $L$ is induced by $\Sigma$ as follows:
\begin{enumerate}
    \item [(CD)] $X \subseteq L$  is compact in  $L$ if and only if $\pi^{-1}(X)$ is $\delta$-definable in $\Sigma$, and $X \subseteq L$  is open in $L$ if and only if $\pi^{-1}(X)$ is $\sigma$-definable in $\Sigma$.
\end{enumerate}

It is known that (CD) is equivalent to 
\begin{enumerate}
    \item[(CD$'$)] Whenever $F \subseteq U \subseteq L$ are such that $F$ is compact and $U$ is open, there is  $D \subseteq \langle S \rangle$ that is definable in $\Sigma$ and $\pi^{-1}(F) \subseteq D \subseteq \pi^{-1}(U)$.
\end{enumerate}
We will not be using this equivalence, so we will leave it to the interested reader.
The the definition good model given in~\cite{BGT} is essentially combines the definition of locally compact model and (CD$'$), so it is in a sense the original definition.

Our definition of locally compact model is not the definition given in~\cite{MW}. (We chose it due to conceptual clarity, and its later use for constructing of Lie models).  To discuss the definition of locally compact models in~\cite{MW},  we introduce two further conditions:
\begin{enumerate}
            \item[(LM1$'$)] (Containing the kernel) $\ker \pi \subseteq S$ 
            \item[(LM2$'$)] (Bounded index) If $A, B \subseteq \langle S \rangle $ are definable in $\Sigma$ and $\ker \pi \subseteq A$, then finitely many left-translates of $A$ by elements in $\langle S  \rangle $ can be used to cover $B$.
    \end{enumerate}

The equivalence between the definition is given by the following lemma:

\begin{lemma} \label{lemma: equivalentLiemodeldefinition}
   Suppose $(G, \Sigma)$ is an $\aleph_1$-saturate expansion of a group, and $S \subseteq G$ is definable in $\Sigma$. Let $L$ be a locally compact group, and $\pi:  \langle S \rangle \to L$ a surjective group homomorphism  continuously definable in $\Sigma$. Then (LM1) is equivalent to (LM1$'$) and  (LM2) is equivalent to (LM2$'$).
\end{lemma}

\begin{proof}
We first show the equivalent between (LM1) and (LM1'), and clearly only the backward implication is needed. Choose a sequence $(U_n)$ of open neighborhood of $\id_G$. Then $\bigcap_n \pi^{-1}(U_n) = \ker \pi$ is a subset of $S$. It suffices to show there is $n$ such that $\pi^{-1}(U_n) \subseteq S$. Suppose this is not the case. By (4), for each $n$, $\pi^{-1}(U_n)$ is a countable union of subsets of $G$ definable in $(G;S)$. Using induction, one can construct a sequence $(D_n)$ such that $D_n\subseteq \pi^{-1}(U_n)$ is definable in $\Sigma$ and 
    $$  \bigcap^n_{i=1} D_i \cap \bigcap_{j>n} \pi^{-1}(U_n) $$
is not a subset of $S$. In particular,  any finite intersection of $(D_n)$ is not a subset of $S$. Finally, the $\aleph_1$-saturation of $(G;\Sigma)$ yields the desired conclusion.

Next we show that (LM2) implies (LM2'). For the forward direction, let $A, B \subseteq \langle S \rangle$ be definable in $\Sigma$, and $\ker \pi \subseteq A$. An argument like in the equivalence between (LM1) and (LM1') shows that there is an open neighborhood $U$ of $\id_L$ such that $\pi^{-1}(U) \subseteq A$. On the other hand, by $\aleph_1$-saturation or Lemma~\ref{lem: Interpolation}, $B$ is contained in $S^n$ for some $n$. This implies that $\pi(B)$ is precompact. It follows that finitely many left-translates of $A$ can be used to cover $B$.

Finally, we show (LM2') implies (LM2). Take $U$ a precompact open neighborhood of $\id_L$, which exists due to the fact that $L$ is locally compact. Then $\pi^{-1}(U)$ is $\sigma$-definable and contains $\ker \pi$. By $\aleph_1$-saturation via Lemma~\ref{lem: Interpolation}, there is $D \subseteq \pi^{-1}(U)$ such that $\ker \pi \subseteq D$. By (LM2'), finitely many left translates of $D$ covers $S$. This implies that finitely many left translates of $U$ covers $\pi(S)$, which implies that $\pi(S)$ is precompact.
 \end{proof}

The following fact~\cite{MW} is the main model-theoretic input that will allow us construct Lie model. We use this as a black box in this paper.

\begin{fact}[Hrushovski's Locally Compact Model Theorem via Massicot--Wagner]\label{fact: locallycompactmodel}
   Suppose $(G, \Sigma)$ is an $\aleph_1$-saturated expansion of a group, and  $S \subseteq G$ is an approximate group definable in $\Sigma$ and definably amenable in $(G, \Sigma)$.  Then there is a locally group $L$, and surjective group homomorphism $\pi: \langle S \rangle \to L $ satisfying  (LM1'), (LM2'), and (CD) with $S$ replaced by $S^4$.
\end{fact} 

To obtain connected Lie models from locally compact model, we need the solution of Hilbert's 5th problem, which is known as the Gleason--Yamabe Theorem~\cite{Gleason,Yamabe}. Here is a version of this result, which is another major black box.

\begin{fact}[Gleason--Yamabe Theorem]\label{fact: Gleason}
Suppose  $L$ is a locally compact group and $U$ is an open neighborhood of $\id_L$. Then there is an open subgroup $L'$ of $L$ and a compact normal subgroup $K$ of $L'$ such that $K \subseteq U$ and the $ L'/K$  is a connected Lie group.
\end{fact}

Suppose $G$ is  a group. Two approximate groups $S$ and $S'$  of $G$ are {\bf equivalent} if each is contained in finitely many left (equivalently, right) translations of the other. It is easy to see that this is indeed an equivalence relation. Moreover, by Lemma~\ref{lem: interativesmallgrowthapproximategroups}, if $S$ is an approximate group, then $S^n$ is an approximate group equivalent to $S$ for all $n$.

\begin{proposition} \label{Prop: ChangingtoAPwithLiemodel}
   Suppose $(G, \Sigma)$ is an $\aleph_1$-saturated expansion of a group,  and $S \subseteq G$ is an approximate subgroup definable in $\Sigma$ and definably amenable in $(G, \Sigma)$. Then there is approximate subgroup $S' \subseteq G$ definable in $\Sigma$ and definably amenable in $(G, \Sigma)$  such that  $S'$ is equivalent to $S$, and  $S'$ has a connected Lie model definable in $\Sigma$. 
\end{proposition}

\begin{proof}
By Lemma~\ref{lemma: equivalentLiemodeldefinition} and Fact~\ref{fact: locallycompactmodel}, $S^4$ has a locally compact model $\pi: \langle S^4 \rangle \to L$. Since $S^4$ is an approximate group equivalent to $S$, it is harmless to replace $S$ with $S^4$ and assume $\pi$ is already the locally compact model of $S$. In particular, using (LM1),  we can choose an open neighborhood $U$ of $\id_L$ such that $\pi^{-1}(U) \subseteq S$. Applying the Gleason--Yamabe theorem, we get an open subgroup $L'$ of $L$ and a compact normal subgroup $K$ of $L'$ such that $K \subseteq U$ and the quotient group $L'/K$ is a connected Lie group.

Now, the group $L'$ is  open in $L$. Hence, $L'$ is also closed in $L$. Therefore, $\pi^{-1}(L')$ is both $\delta$-definable and $\sigma$-definable in $\Sigma$. By $\aleph_1$-saturation via Lemma~\ref{lem: Interpolation}, we learn that $\pi^{-1}(L')$ is definable in $\Sigma$. Set $S' = S \cap  \pi^{-1}(L')$. It remains to verify that $S'$ has the desired property.

As both $S$ and $\pi^{-1}(L)$ are symmetric and contain $\id_G$, so is  $S'$. Moreover, $S'$ contains $\ker \pi$ as a subset, and $(S')^2$ is definable in $\Sigma$. It follows from (LM2$'$) in Lemma~\ref{lemma: equivalentLiemodeldefinition} that finitely many left-translates of $S'$ covers $(S')^2$. We have verified that $S'$ is an approximate group.
Note that both $S$ and $S'= S \cap \pi^{-1}(L')$ contains $\ker \pi = \pi^{-1}(\id_L)$. Hence, it again follows from (LM2$'$) in Lemma~\ref{lemma: equivalentLiemodeldefinition} that $S$ and $S'$ are equivalent approximate groups. 

The set $S' = S \cap \pi^{-1}(L')$ is definable in $\Sigma$ as it is the intersection of two sets definable in $\Sigma$.  Let $\mu$ be the finitely additive measure on $\langle S \rangle $ witnessing the fact that $S$ is definably amenable in $(G, \Sigma)$, that is $\mu$ is left-invariant,  $\mu(S)=1$, and every subset of $\langle S \rangle$ definable in $\Sigma$ is $\mu$-measurable. 
Because $S' \subseteq \langle S \rangle$, $S'$ is definable in $\Sigma$, and $S'$ is an equivalent approximate group to $S$, we have $0< \mu(S')< \infty$. Scaling $\mu$ appropriately, we get a measure witnessing the fact that $S'$ is definably amenable in $(G, \Sigma)$

Now, as $\pi$ is a locally compact model of $S$, by (LM1), $\pi(S')= \pi(S) \cap L'$ contains an open subset of $L'$.   Since $L'$ is connected,  $L' = \langle \pi(S') \rangle $. Hence, $\pi|_{\langle S' \rangle} : \langle S' \rangle \to L'$ is surjective. We now define 
\[
\pi': \langle S' \rangle \to L'/K,\qquad \pi'= \phi \circ\pi|_{\langle S' \rangle}
\]
with $\phi: L' \to L'/K$ the quotient map. By construction $\ker \pi' = \pi^{-}(H)$, and recall that $H \subseteq U \cap  L'$ with $\pi^{-1}(U) \subseteq S$. Hence, $ \ker \pi' \subseteq S'$. It follows that $\pi'$ satisfy (LM1') from Lemma~\ref{lemma: equivalentLiemodeldefinition}. Since, $\ker \pi \subseteq \ker \pi'$, it follows that $\pi'$ satisfy (LM2') from Lemma~\ref{lemma: equivalentLiemodeldefinition} too. The group homomorphism $\pi'$ satisfy (CD) because $\pi$ satisfy (CD) and $\phi$ is continuous. Thus, by Lemma~\ref{lemma: equivalentLiemodeldefinition}, $\pi'$ is a locally compact model. Finally, $L'/K$ is a connected Lie group, so $\pi'$ is a connected Lie model.
\end{proof}

We end with a remark about plausible alternative approach:
\begin{remark}
 
Here, we are using the Massicot--Wagner version of Hrushovski's locally compact model theorem. This is easier for us as it is compatible with the setting with measure. However, one can also use the original version of the theorem by Hrushovski~\cite{Hrushovski}. For that one will need to verify in addition that the collection of null sets with respect to the pseudo-Haar measure form an S1-ideal, which is possible.

    In his thesis~\cite{thesis}, Carolino provides an alternative way to link pseudo-measurable approximate groups to connected Lie group. Instead of using definably amenable approximate groups as we do here, he works with a notion of pseudo-open approximate groups and follow the strategy in~\cite{BGT} to get a suitable version of Hrushovski's Lie model theorem. Even though we believe his results are by-and-large correct, the thesis is not peer-reviewed and we think some details must be further clarified. For example, in Lemma~6.8 of~\cite{thesis}, it was not verified that $\pi^{-1}(G')$ is pseudo-open and deduce that $\Tilde{A}$ is an ultra-approximate group in the sense he defined. This is particular important for us because without the knowledge of pseudo-open, the product sets might fail to be measurable with respect to a pseudo-Haar measure. If such issues can be addressed, Carolino's thesis will provide an alternative path to a connected Lie model.
\end{remark}

\section{Small growth above Lie models}

In the previous sections, we construct a suitable approximate group $S$ based on $A$ and $B$. However, the Brunn--Minkowski growth $\BM(S,S)$ may be much larger than $\BM(A, B)$. The purpose of the section is to find two sets  $A',B' \subseteq \langle S \rangle$ so that $\BM(A',B')$ is very close to $\BM(A,B)$. This step is the essential reason we also need to treat the asymmetric case.

Let $G$ be a pseudo-compact group built from a sequence $(G_n)$ of compact groups and a nonprincipal ultrafilter $\ult$. We say that $G$ is {\bf pseudo-connected} if each $G_n$ is connected. 

The primary use of the connectedness assumption is the  inequality below by Kemperman~\cite{Kemperman}. Note that this inequality does not hold when the group is disconnected by choosing the set to be the identity component.

\begin{fact}[Kemperman Theorem] \label{fact: Kemperman}
Suppose $G$ is a connected unimodular locally compact groups equipped with a Haar measure $\mu$. If $A, B\subseteq G$ are such that $A$,$B$, and $AB$ are measurable, then $\mu(AB) \geq \min\{\mu(A)+\mu(B),\mu(G)\}$.
\end{fact}

The following lemma is the version of Kemperman theorem for pseudo-connected pseudo-compact groups.

\begin{lemma}\label{lem: pseudo-Kemperman}
Suppose $G$ is pseudo-compact and pseudo-connected with a pseudo-Haar measure $\mu$. If $A, B\subseteq G$ are such that $A$,$B$, and $AB$ are pseudo-measurable, then $\mu(AB) \geq \min\{\mu(A)+\mu(B),\mu(G)\}$.
\end{lemma}
\begin{proof}
Suppose $G$ is the pseudo-compact group built from a sequence $(G_n)$ of compact and connected groups and a nonprincipal ultrafilter $\ult$. 
Let $(A_n)$ and  $(B_n)$, and sequences such that $A_n,B_n\subseteq G_n$ is measurable for each $n$. Applying the Kemperman theorem and taking limit we get the desired conclusion. 
\end{proof}

We now prove the main result of this section.

\begin{proposition} \label{Prop:SmallgrowthaboveLiemodel}
Suppose  $(G, \Sigma)$ is an expansion of a pseudo-connected pseudo-compact group such that every $D \subseteq G$ definable in $\Sigma$ is pseudo-measurable. Let $A, B \subseteq G$ be definable in $\Sigma$, commensurable to one another, and infinitesimal compared to $G$,  let  $S \subseteq G$ be an approximate group definable in $\Sigma$ and commensurable to both $A$ and $B$, and assume that $A$ can be covered by finitely many left-translates of $S$, and $B$ can be covered by finitely many right-translates of $S$. Then for all $\epsilon>0$, there are  $A', B' \subseteq \langle S \rangle $ definable in $\Sigma$, commensurable to $A$, $B$, and $S$ such that $\BM(A',B') \leq \BM(A,B)+\epsilon$. 
\end{proposition}

\begin{proof}
 Let $\mu$ be a pseudo-Haar measure on $G$ such that $0< \mu(A), \mu(B), \mu(S)< \infty$, so $\mu(G)= \infty$. 
%It suffices to show that $\BM(A,B) \geq r$ with
% $$r := \inf \{ \BM(A', B') : A', B' \subseteq \langle S \rangle \text{ definable in } \Sigma \text{ and commensurable to } S  \}.$$ 
As $G$ is a pseudo-connected and pseudo-compact group, applying Lemma~\ref{lem: pseudo-Kemperman}, we get 
$$  \mu(AB) \geq \min\{\mu(A)+\mu(B), \mu(G)\}.$$
As $\mu(G) =\infty$, we have $\BM(A,B) \geq 1$.% Hence, we can further assume $r \geq 1$, as we are done otherwise. 

For each left coset $\alpha \in G/\langle S \rangle$, choose a representative $g(\alpha) \in G$, and for each right coset $\beta \in \langle S \rangle \backslash G$, choose a representative $g(\beta) \in G$.
By the relationship between $A$, $B$, and $S$, there are only finitely many left cosets of $\langle S \rangle$ intersecting $A$ nonemptily, and a similar claim holds for $B$ with right cosets of $\langle S \rangle$. Therefore, 
\begin{equation}\label{eq: integral formula in G}
\mu(A)=\sum_{\alpha\in G/\langle S \rangle} \mu_{\langle S \rangle}(g(\alpha)^{-1}A\cap \langle S \rangle),\quad \mu(B)=\sum_{\beta\in \langle S \rangle\backslash G} \mu_{\langle S \rangle}(B g(\beta)^{-1}\cap \langle S \rangle).
\end{equation}
For simplicity, for every $\alpha\in G/ \langle S \rangle$ and $\beta\in  \langle S \rangle\backslash G$, we write
\[
A_\alpha := g(\alpha)^{-1}A\cap \langle S \rangle,\quad\text{and}\quad B_\beta:=B g(\beta)^{-1}\cap \langle S \rangle. 
\]
As $A,B$ are definable in $\Sigma$, for each  $\alpha\in G/ \langle S \rangle$ and $\beta\in  \langle S \rangle\backslash G$, the sets $A_\alpha, B_\beta \subseteq \langle S \rangle$ are also definable in $\Sigma$. It is harmless to delete such $A_\alpha, B_\beta$ when they have $\mu$-measure $0$ and replacing $A$ and $B$ with subsets accordingly. Hence, we can assume for each  $\alpha\in G/ \langle S \rangle$ and $\beta\in  \langle S \rangle\backslash G$, the sets $A_\alpha, B_\beta$ are commensurable to $S$.

Let $r=\BM(A,B)$. Then $r\geq 1$. We will show that there are $\alpha\in G/ \langle S \rangle$ and $\beta\in  \langle S \rangle\backslash G$ such that $\BM(A_\alpha,B_\beta)\leq r + \varepsilon$ for all positive $\varepsilon$. From now on we assume that there is $\varepsilon>0$ such that for all $\alpha\in G/ \langle S \rangle$ and $\beta\in  \langle S \rangle\backslash G$, $\BM(A_\alpha,B_\beta)> r+\varepsilon$. Equivalently, we have
\begin{equation}\label{eq: assumption}
\left(\frac{\mu_{\langle S \rangle}(A_\alpha)}{\mu_{\langle S \rangle}(A_\alpha B_\beta)}\right)^{\frac{1}{r+\varepsilon}}+ \left(\frac{\mu_{\langle S \rangle}(B_\beta)}{\mu_{\langle S \rangle}(A_\alpha B_\beta)}\right)^{\frac{1}{r+\varepsilon}}\leq 1.
\end{equation}

We now define two probability measures on $G/\langle S \rangle$ and $\langle S \rangle\backslash G$ based on the shapes of $A$ and $B$ respectively as follows. For $\alpha\in G/\langle S \rangle$, define
\[
\pr_A(\alpha)=\frac{\mu_{\langle S \rangle}(A_\alpha)}{\mu(A)}.
\]
And similarly, for $\beta\in \langle S \rangle\backslash G$, let 
\[
\pr_B(\beta)=\frac{\mu_{\langle S \rangle}(B_\beta)}{\mu(B)}.
\]

Choose $\alpha$ from $G/\langle S \rangle$ randomly with respect to $\pr_A$. We obtain that 
\begin{equation}\label{eq: about to use Holder}
\mathbb{E}_{\pr_A(\alpha)}\left(\frac{\mu_{\langle S \rangle}(A_\alpha)}{\mu_{\langle S \rangle}(A_\alpha B_\beta)}\right)^{\frac{1}{r+\varepsilon}}=\frac{1}{\mu(A)}\sum_{\alpha \in G/\langle S \rangle}\frac{\left(\mu_{\langle S \rangle}(A_\alpha)\right)^{\frac{r+\varepsilon+1}{r+\varepsilon}}}{\left(\mu_{\langle S \rangle}(A_\alpha B_\beta)\right)^{\frac{1}{r+\varepsilon}}}. 
\end{equation}
On the other hand, by H\"older's inequality (with exponents $(r+\varepsilon)/(r+\varepsilon +1)$ and $1/(r+\varepsilon +1)$) we get
\[
\left(\sum_{\alpha\in G/\langle S \rangle}\mu_{\langle S \rangle}(A_\alpha)\right)^{\frac{r+\varepsilon+1}{r+\varepsilon}} \leq \sum_{\alpha\in G/\langle S \rangle} \frac{\mu_{\langle S \rangle}(A_\alpha)^{\frac{r+\varepsilon+1}{r+\varepsilon}}}{\mu_{\langle S \rangle}(A_\alpha B_\beta)^{\frac{1}{r+\varepsilon}}}\left(\sum_{\alpha\in G/\langle S \rangle}\mu_{\langle S \rangle}(A_\alpha B_\beta)\right)^{\frac{1}{r+\varepsilon}}. 
\]
Using \eqref{eq: integral formula in G} and the fact that
\[
A_\alpha B_\beta = \left(g(\alpha)^{-1}A\cap \langle S \rangle \right)\left(B g(\beta)^{-1}\cap \langle S \rangle\right) \subseteq \left(g(\alpha)^{-1}A B g(\beta)^{-1}\right)\cap \langle S \rangle,
\]
together with the unimodularity of $G$ we conclude that
\[
\sum_{\alpha\in G/\langle S \rangle}\mu_{\langle S \rangle}(A_\alpha) = \mu(A), \quad\text{and}\quad \sum_{\alpha\in G/\langle S \rangle}\mu_{\langle S \rangle}(A_\alpha B_\beta)\leq \mu(ABg(\beta)^{-1})=\mu(AB). 
\]
Finally, by using \eqref{eq: about to use Holder}, we get
\begin{equation}\label{eq: the A part}
    \mathbb{E}_{\pr_A(\alpha)}\left(\frac{\mu_{\langle S \rangle}(A_\alpha)}{\mu_{\langle S \rangle}(A_\alpha B_\beta)}\right)^{\frac{1}{r+\varepsilon}}\geq \left(\frac{\mu(A)}{\mu(A B)}\right)^{\frac{1}{r+\varepsilon}}.
\end{equation}

Similarly by choosing $\beta$ from $\langle S \rangle \backslash G$ randomly with respect to $\pr_B$, we will get
\begin{equation}\label{eq: the B part}
    \mathbb{E}_{\pr_B(\beta)}\left(\frac{\mu_{\langle S \rangle}(B_\beta)}{\mu_{\langle S \rangle}(A_\alpha B_\beta)}\right)^{\frac{1}{r+\varepsilon}}\geq \left(\frac{\mu(B)}{\mu(A B)}\right)^{\frac{1}{r+\varepsilon}}.
\end{equation}
Inequalities \eqref{eq: the A part} and \eqref{eq: the B part} together with the Fubini and \eqref{eq: assumption} give us
\begin{align*}
&\,\left(\frac{\mu(A)}{\mu(AB)}\right)^{\frac{1}{r+\varepsilon}}+\left(\frac{\mu(B)}{\mu(AB)}\right)^{\frac{1}{r+\varepsilon}}\\
\leq&\, \mathbb{E}_{\pr_A(\alpha)}\mathbb{E}_{\pr_B(\beta)}\left[\left(\frac{\mu_{\langle S \rangle}(A_\alpha)}{\mu_{\langle S \rangle}(A_\alpha B_\beta)}\right)^{\frac{1}{r+\varepsilon}}+ \left(\frac{\mu_{\langle S \rangle}(B_\beta)}{\mu_{\langle S \rangle}(A_\alpha B_\beta)}\right)^{\frac{1}{r+\varepsilon}} \right]\leq1.
\end{align*}
This contradicts the choice of $r$. 
\end{proof}

\section{Density functions of definable sets over Lie models}

When $H$ is a locally compact group, and $\pi: H \to L$ is a continuous and surjective group homomorphism, a left Haar measure $\mu$ of  $H$ and a left Haar measure $\lambda$ of $L$ are linked together by the corresponding left Haar measure $\nu$ of $\ker \pi$ via a quotient integral formula. This allows us to relate the measure of a measurable set $A \subseteq G$ and the measure of its image $\pi(A)$ using the density function $f_A(x) = \nu ( \ker \pi \cap x^{-1}A  )$. 

There is no similar measure on the kernel of a locally compact model or Lie model. Nevertheless, we will now show there are still  density functions which are   well-behaved enough for our later purpose.

Suppose $H$ is a group equipped with a left-invariant measure $\mu$, the set $A \subseteq H$ is measurable, $\pi:H \to L$ is a  surjective group homomorphism, and the group $L$ is equipped with a left-invariant measure $\lambda$. We say that $f_A: H \to \RR^{\geq 0}$ is a {\bf density function} for $A$ with respect to $\pi$ if for all measurable $X \subseteq L$, the set $\pi^{-1}(X)$ is $\mu$-measurable, and 
$$ \mu( A \cap \pi^{-1} (X)) = \int_X f_A \d\lambda. $$
We have a similar definition replacing left by right.

It is clear from the above definition that for densities to exist, the measure on $H$ must be sufficiently rich. The following lemma verifies that this is the case when we are handling locally compact models.

\begin{lemma} \label{lem: Link between measures}
    Suppose $(G, \Sigma)$ is an $\aleph_1$-saturated expansion of a group,  $S \subseteq G$ is an approximate group definable in $\Sigma$, $S$ is definably amenable in $(G, \Sigma)$ witnessed by a left-invariant measure $\mu$  on $\langle S \rangle$, and $S$ has a locally compact model  $\pi: \langle S \rangle \to L$ continuously definable in $\Sigma$. Then $\pi^{-1}(X) \subseteq \langle S \rangle$ is $\mu$-measurable whenever $X \subseteq L$ is measurable. Moreover, if $\lambda$ is the pushforward of $\mu$ on $L$, given by 
    $$ \lambda(X):= \mu(\pi^{-1}(X)) $$
    for measurable $X \subseteq L$, then $\lambda$ is a left Haar measure on $L$.

\end{lemma}
\begin{proof}

We first note that if $X \subseteq L$ is Borel, then $\pi^{-1}(X)$ is $\mu$-measurable. Indeed, it is enough to check for open $X \subseteq L$ , that $\pi^{-1}(X)$ is $\mu$-measurable. This is the case because from the condition that $\pi$ is definable in $\Sigma$, $\pi^{-1}(U)$ is open.

Now define a Borel measure $\Tilde{\lambda}$ on $L$ as follows. When $X \subseteq L$ is Borel, set
$$ \Tilde{\lambda}(X) = \mu( \pi^{-1} (X)). $$
The $\sigma$-additivity of $\Tilde{\lambda}$ on Borel sets and $\Tilde{\lambda}(\emptyset)=0$ follow from the corresponding facts for $\mu$.
Moreover, $\Tilde{\lambda}$ is the completion of its restriction to the collection of Borel subsets.

It will turn out that $\Tilde{\lambda}$ is a Haar measure, but will postpone this proof and show why the desired conclusions follow from it. By the uniqueness of Haar measure, we get constant $\alpha$ such that $ \lambda(X) = \alpha \Tilde{\lambda}(X)$ for all measurable $X$. When $B$ is Borel, from our construction, this already gives us that $\lambda(B) = \alpha \mu(\pi^{-1}(B))$.  If $X$ is measurable, and $\lambda(X)=0$, then $X$ is a subset of Borel $B \subseteq L$ with $\lambda(B)=0$ as $\lambda$ is the completion of its restriction to the Borel sets. For such $X$, the set $\pi^{-1}(X)$ is a subset of $\pi^{-1}(B)$ which has $\mu(\pi^{-1}(B))=0$, so $\pi^{-1}(X)$ is measurable and has zero $\mu$ measure by the completeness of $\mu$. For a general measurable $X \subseteq L$, a standard argument produces Borel $B\subseteq L$ such that $\lambda(X \triangle B)=0$. It then follows from the earlier cases that $\pi^{-1}(X)$ is measurable and $ \lambda(X) = \alpha \mu(\pi^{-1}(X)). $

 Finally, we verify that $\Tilde{\lambda}$ is a left Haar measure of $L$. It is clear that $\Tilde{\lambda}$ is left-invariant. So we only need to show that $\Tilde{\lambda}$ is a Radon measure. By (LM1) of the definition of locally compact model, there is an open neighborhood $U$ of $\id_L$ such that $\pi^{-1}(U) \subseteq S$. Then for each $x \in U$, $xU$ is is open neighborhood with 
$$\Tilde{\lambda}(U)  = \mu( \pi^{-1}( U)) < \mu (S) =1.$$
From this we established the locally finiteness (RM1) property of $\Tilde{\lambda}$. It remains to show that $\Tilde{\lambda}$ is inner regular for open sets and outer regular. We will show this through several claims.

\medskip

\noindent {\bf Claim 1:} If $D \subseteq \langle S \rangle$ is definable in $\Sigma$, then $\pi(D)$ is compact in $L$.

\medskip

\noindent{\it Proof of Claim 1:} Recall that the definability of $\pi$ implies that $C\subseteq L$ is closed if and only if $\pi^{-1}(C)$ is $\delta$-definable. So we need to show  $\pi^{-1}( \pi(D)) = D (\ker \pi)$ is $\delta$-definable. As $\ker \pi = \pi^{-1}(\{\id_L \})$ is $\delta$-definable, we obtain a decreasing sequence $(E_n)$ definable in $\Sigma$ such that 
$$ \ker \pi = \bigcap_{n} E_n.$$
Clearly, we have $D (\ker \pi) \subseteq \bigcap_{n} DE_n$. We will now show that $  \bigcap_{n} DE_n \subseteq D (\ker \pi) $. Take $a \in \bigcap_{n} DE_n$. Then $a= d_ne_n$ with $d_n \in D$ and $e_n \in E_n$. Hence,   $D \cap (a(E_n)^{-1}) $ containing $a(e_n)^{-1}=d_n$ is nonempty. By $\aleph_1$-saturation, $D \cap \bigcap_n a(E_n)^{-1}$ is nonempty. It follows that there is $d\in D$ and $e^{-1} \in \bigcap_n (E_n)^{-1}$ such that $d =a e$. Thus, $a = de$ is in $D \ker \pi$.

Finally, by $\aleph_1$-saturation, $D$ is a subset of $S^n$ for some $n$. Hence, $\pi(D)$ is a subset of $(\pi(S))^n$, which is precompact. Therefore, $\pi(D)$ is compact.
\hfill$\bowtie$
\medskip

\noindent {\bf Claim 2:} Suppose $U \subseteq L$ is open. Then $U$ is a countable union of compact sets. In particular, $\Tilde{\lambda}$ has inner regularity for open sets. 

\medskip

\noindent{\it Proof of Claim 2:}  Recall that the definability of $\pi$ implies that $U\subseteq L$ is open if and only if $\pi^{-1}(U)$ is $\sigma$-definable in $\Sigma$. Hence, $\pi^{-1}(U) = \bigcup_n D_n$ with $D_n$ definable in $\Sigma$ for each $n$. It then follows that  
$$U = \bigcup_n \pi( D_n).$$
From claim 1, it follows that $U$ is a countable union of compact subsets of $L$. 
\hfill$\bowtie$

\medskip

\noindent {\bf Claim 3:} Suppose $U \subseteq L$ is open and precompact, and $B\subseteq U$ is Borel. Then $\Tilde{\lambda}$ is both inner regular and outer regular on $U$. 

\medskip

\noindent{\it Proof of Claim 3:} We will prove using induction on the Borel complexity of $B$. When $B$ is open, the inner regularity follows from Claim~2 and the outer regularity is immediate. If $B = U \setminus B_1$ and we have proven the statement for $B_1$, then the statement for $B$ will simply follows by switching inner and outer regularity. Suppose $B= \cup_n B_n$, and we have proven the statement for $B_n$ for each $n$. Fix $\epsilon>0$. Let  
$B_N = \bigcup_{n=1}^N B_n$ such that $\Tilde{\lambda}(B \setminus B_n)< \epsilon/2$. For $n \in [N]$, using the induction hypothesis, we choose compact $K_n \subseteq B_n$ such that $\Tilde{\lambda} (B_n \setminus K_n) < (\epsilon/2) (1/2)^n$. Then, it is easy to see that $\Tilde{\lambda}B \setminus \bigcup_{n=1}^N K_n$. This shows the inner regularity of $\Tilde{\lambda}$ for $B$. A similar approximation argument shows the outer regularity of $\Tilde{\lambda}$ for $B$.  
\hfill$\bowtie$
\medskip

\noindent {\bf Claim 4:} Suppose $B \subseteq L$ is Borel. Then $\Tilde{\lambda}$ is outer regular on $B$. 

\medskip

\noindent{\it Proof of Claim 4:}   From Claim 3, it is enough to establish that $L$ is a countable union of open an precompact subsets of it. As $\pi$ is a locally compact model, there is a neighborhood $U$ of $\id_L$ such that $U \subseteq \pi S$ and   $\pi(S)$ is precompact. Hence, this set $U$ is open and precompact. Note that $L = \bigcup_{n} \pi(S^n)$. For each $n$, $\pi(S^n)$ is precompact, and hence can be covered by finitely many translation of $U$. Therefore, $L$ can be covered by countably many translates of $U$, which completes the proof. 
\hfill$\bowtie$
\medskip

This completes the proof.
\end{proof}

We recall some concept from analysis. Let $\Omega$ be a topological space equipped with a measure $\lambda$. A measure $\kappa$ with the same $\sigma$-algebra as $\lambda$ is {\bf absolutely continuous} with respect to $\lambda$ if and only if for each measurable $X \subseteq L$, we have $\mu(X)=0$ whenever $\lambda(X)=0$. 

% Moreover, if $A$ is measurable in $S$, let  $\lambda_A$ be the measure on $L$ defined by 
%     $$  \lambda_A(X) = \nu(A \cap \pi^{-1}(X))$$
%     for measurable $X\subseteq L$. Then $\lambda_A$ is absolutely continuous continuous on $L$.

We need the following fact:

\begin{fact}[Radon--Nikodym Theorem] \label{RadonNikodym}
Suppose $\kappa$ and $\lambda$ are $\sigma$-finite Borel measure on $L$ and $\kappa$ is absolutely continuous with respect to $\lambda$ (i.e, for all Borel $X\subseteq L$, $\lambda(X) =0$ implies $\kappa(X)=0$). Then there is a $\lambda$-measurable function $f$ such that for all Borel $X\subseteq L$ we have
$$ \kappa(X) =\int_{X} f \d\lambda x. $$
\end{fact}

We now prove the promised existence:

% \begin{lemma}
% Suppose $(H, \Sigma, \nu)$ is a measured group with unimodular Lie model $\pi: H \to L$, $\lambda$ is the Haar measure of $L$, and $A$ is $\nu$-measurable such that $\pi(A)$ is contained in a compact set. Then the density function $f_A$ of $A$ with respect to $\pi$ exists and is compactly supported.
% \end{lemma}

\begin{proposition} \label{prop: existenceanduniqueness}
    Suppose $G$ is a group equipped with a complete left invariant measure $\mu$, and $S\subseteq G$ is an approximate subgroup with $\pi: \langle S \rangle \to L$ is a locally compact model. If $A \subseteq \langle S \rangle$ is definable in $(G,S)$, then $A$ has a density function $f_A$ with respect to $\pi$ which is a.e. bounded. Moreover, if $g_A$ is another density function of $A$ with respect to $\pi$, then $f_A(x) = g_A(x)$ for a.e. $x\in L$.

\end{proposition}

\begin{proof}
Define the Borel measure $\kappa_A$ on $L$ by setting
$$ \kappa_A(X) := \nu( A \cap \pi^{-1} X)$$
whenever $X \subseteq L$ is Borel. From the definition of Lie model, one has $0\leq \kappa_A(X) \leq  \nu (X) $. In particular, $\kappa_A$ is absolutely continuous with respect to $\lambda$. As $L$ is connected, $\lambda$ is $\sigma$-finite, and so is $\kappa_A$. Thus, we can  apply Fact~\ref{RadonNikodym} to get the desired conclusion. The last statement is immediate.
\end{proof}

The fact that there is no single canonical density function creates more problem than it might initially seem. We will later need to study the relationship between the density functions $f_{A}$, $f_{B}$, and $f_{AB}$ for definable $A, B \in \langle S \rangle$. As $AB$ is generally a uncountable union of left-translates of $B$ and a  uncountable union of right-translates of $A$, such relationships becomes unclear.  We will introduce the variant notion of approximate density functions, which is canonical.

Suppose $H$ is a group equipped with a left invariant measure $\mu$, the set $A \subseteq H$ is measurable,  $\pi: H \to L$ is a surjective group homomorphism, and the group $L$ equipped with a left-invariant measure $\lambda$. Suppose moreover, that $d$ is a left-invariant metric on $L$ such that $\lambda$ is a nonzero Radon measure with respect to the topology generated by $d$ and open balls has finite positive measure.  For $\epsilon>0$, the {\bf $(d,\epsilon)$-density} function $f^{d,\epsilon}_A$ with respect to $\pi$  is given by 
$$ f_A^{d,\epsilon}(x) := \frac{\mu (A \cap \pi^{-1}(B_\epsilon(x) )} { \lambda (B_{\epsilon}(x)) }. $$
for $x \in L$. It is immediate to see that if $f_A$ is a density function of $A$ with respect to $\pi$ then we have 
$$ f_A^{d,\epsilon}(x) = \EE_{t\in B_\epsilon(x)} f_A(t).$$
Note that the above equation does not depends on the choice of the density function $f_A$.

Earlier, we have seen that approximate density functions can be understood in term of density functions. Next, we will see that we can go in the other direction in the setting of Lie model.

Let $(\Omega, d)$ be a metric space, and $\mu$ is a Radon measure on $\Omega$ such that open balls have positive and finite measure. We call $x\in \Omega$ a {\bf Lebesgue point} of the function $f$ if
$$ \lim _{\epsilon \to 0^+} \EE_{t \in B_\epsilon(x)} |f(x)-f(t)|=  0. $$
We say that $(\Omega, d)$  is {\bf locally doubling} if there $r>0$ and a constant $K$ such that for all $x \in \Omega$ and $0< \epsilon< r$, we have 
$$ \frac{\lambda(B_{2\epsilon}(x)}{\lambda(B_\epsilon(x))} < K.  $$

The following fact~\cite[Section 3.4]{Sobolev} from real/harmonic analysis is the key to our approach. The theorem in its stronger form, that almost every point is a Lebesgue point of a locally integrable function $f$, can be proved as a consequence of the weak-$L^1$ estimates for the Hardy--Littlewood maximal function.

\begin{fact}[Lebesgue differentiation theorem for locally doubling metric space] \label{fact: Lesbeguedifferentiation}
Let $(\Omega, d)$ be a metric space, and $\mu$ is a Radon measure on $\Omega$ such that open balls have positive and finite measure. Suppose $(\Omega, d)$ is locally doubling, and $f$ is a $\lambda$-integrable  function on $\Omega$. Then $\lambda$-a.e. $x \in \Omega$ is a Lebesgue point of $f$.
\end{fact}

We now prove the second main result of this section:

\begin{proposition} \label{prop: good approximation}
    Suppose $H$ is a group equipped with a left invariant measure $\mu$, $\pi: H \to L$ is a surjective group homomorphism, and $L$  is a connected Lie group equipped with a left Haar measure $\lambda$. Suppose moreover, $d$ is the distance function induced by a left-invariant Riemannian metric $L$. Let $f_A$ be a density function for $A$ with respect to $\pi$, and for $\epsilon>0$ let $f^{d, \epsilon}_A$ be the $(d,\epsilon)$-density function for $A$ with respect to $\pi$. Then
     for  a.e. $x \in L$, we have
     $$ f_A(x)= \lim_{\epsilon \to 0^+}  f^{d, \epsilon}(x).  $$
     Similar statements hold replacing left invariant metric and Haar measure with right invariant metric and Haar measure.
\end{proposition}

\begin{proof}
    The desired conclusion will follow from the two claims below and Fact~\ref{fact: Lesbeguedifferentiation}. \smallskip
    
   \noindent  {\bf Claim 1.} The metric space $(L, d)$ is locally doubling.\medskip

   \noindent \textit{Proof of Claim 1.}
     Note that
    $$ \lim_{r \to 0^+} \frac{\lambda(B_{\id_L, 2r})}{\lambda(B_{\id_L, r})}  = 2^n.  $$
    where $n$ is the dimension of $L$.
    Hence, we can choose $K= 2^n+1$ and $r_0$ sufficiently small such that $\lambda(B_{\id_L, 2r})/\lambda(B_{\id_L, r})< K$. Using the fact that both $\lambda$ and $d$ are left-invariant, we arrive at the desired conclusion.  \hfill$\bowtie$ 
    \medskip

   \noindent {\bf Claim 2.}  If $x \in L$ is a Lebesgue point of $f_A$, then 
$$ \lim_{\epsilon \to 0^+} f_A^{d, \epsilon}(x) = f_A(x). $$

\noindent \textit{Proof of Claim 2.}
  Note that $f^{d, \epsilon}(x) -f(x) = \EE_{t \in B_\epsilon(x)} (f(t)-f(x)) $. Hence, the conclusion follows from the definition of Lebesgue points. \hfill$\bowtie$ 
  \medskip
  
The proof for the last statement for the right measures and right metrics is similar.
\end{proof}

\section{Small growth in Lie models}

 In this section we will prove there are sets in the Lie model with small measure growth. We starts by noting some additional properties of the Lie model in the situation we are interested in.

\begin{proposition}
    Suppose $(G, \Sigma)$ is an expansion of a pseudo-connected and pseudo-compact group such that every $D \subseteq G$ definable in $\Sigma$ is pseudo-measurable,   $S\subseteq G$ is an approximate group definable in $\Sigma$ such that $0< \mu(S)< \mu(G) =\infty$ for some pseudo-Haar measure $\mu$, and $\pi: \langle S \rangle \to L$ is a connected Lie model of $S$ continuously definable in $\Sigma$. Then $L$ is a unimodular noncompact connected Lie group.
\end{proposition}
\begin{proof}
We first show that $L$ is unimodular. Without loss of generality, we can assume that $(G, \Sigma)$ is generated by the collection of subsets of $G$ definable in $\Sigma$. Then by Lemma~\ref{lem: Saturatedofgeneratedexpansion}, $(G, \Sigma)$ is $\aleph_1$-saturated. Let $\mu$ also denote its restriction to $\langle S \rangle$. Scaling $\mu$ by a constant, we can arrange that $\mu(S)=1$. Hence, by Lemma~\ref{lem: Definablyamenablearise}, $\mu$ witness that $S$ is a definably amenable approximate group. Thus, we are in the setting to apply Lemma~\ref{lem: Link between measures}. In particular, the pushforward of $\mu$ on $L$ is    a Haar measure $\lambda$ on $L$. The pseudo-Haar measure is right invariant by Lemma~\ref{invarianceofpseudoHaar}. It follows that 
$\lambda$ is also right invariant. Thus $L$ is unimodular.

Next we show that $L$ is noncompact. Suppose to the contrary that this is not the case. By (LM1) in the definition of locally compact model, there is an open neighborhood $U$ of $\id_L$ such that $\pi^{-1}(U) \subseteq S$. Since $L$ is connected, applying the Kemperman inequality (Fact~\ref{fact: Kemperman}), we have
$$ \lambda(U^{n+1}) \geq \min\{ \lambda(U^n) + \lambda(U), \lambda(L)\}.  $$
As $L$ is compact, $\lambda(U)< \infty$. Hence, there is $n$ such that $\lambda(U^n) \geq \lambda(L)$, which implies that $U^{n+1}=L$. For such $n$, we see that $S^{n+1} = \langle S \rangle$.
Since  $S$ is an approximate group, we then learn that $\langle S \rangle$ is contained in contained in  finitely many translates of $S$. So $\mu(\langle S \rangle) < \infty$. On the other hand $S$ is infinitesimal compared to $G$, Lemma~\ref{lem: pseudo-Kemperman} applies and, for all $n$,
\[
\mu(S^{n+1})\geq \mu(S^n) + \mu(S).
\]
This results in a contradiction.
\end{proof}

 Recall that for a given function $f:G\to\RR$, the {\bf superlevel set} of $f$ is 
\[
\LL_f^+(t):=\{g\in G\mid f(g)\geq t\}. 
\]

\begin{lemma} \label{lem: ReplacementBMinkernel}
        Suppose $G$ is a group equipped with an invariant measure $\mu$, and $S \subseteq G$ is an approximate group definably amenable with respect to $\mu$ such that $S$ has a connected Lie model $\pi: \langle S \rangle \to L$.
    Let $A, B \subseteq \langle S \rangle$ be $\mu$-measurable such that $\pi(A), \pi(B)$ have positive measure and are contained  compact sets. Let $f_A$, $f_B$, and $f_{AB}$ be the $\lambda$-a.e.\!  density of $A$, $B$, and $AB$.  For all constant $\alpha, \beta>0$, there are $\sigma$-compact $X_\alpha, Y_\beta \subseteq L$ such that the following holds:
    \begin{enumerate}
        \item $ X_\alpha \subseteq \LL^+_{f_A}(\alpha)$ and $\lambda(X_\alpha) = \lambda( \LL^+_{f_A}(\alpha) )$ 
        \item $Y_\beta \subseteq \LL^+_{f_B}(\beta))$ and  $\lambda(Y_\beta) = \lambda(  \LL^+_{f_B}(\beta))$ 
        \item $X_\alpha Y_\beta$ is $\sigma$-compact,
        \item With $\gamma =\max\{\alpha, \beta\}$, we have $\lambda( X_\alpha Y_\beta \setminus \LL^+_{f_{AB}}(\gamma) ) =0$. 
    \end{enumerate}
\end{lemma}

\begin{proof}
We first consider the case where $\alpha \leq  \beta$.
Using the inner regularity of the Haar measure on $L$, choose  $\sigma$-compact $X_\alpha \subseteq \LL^+_{f_A}(\alpha)$ such that $\lambda(X_\alpha) = \lambda(\LL^+_{f_A}(\alpha)) $. Let $d: L^{[2]} \to \RR^{\geq 0}$ be the distance function induced by a left-invariant Riemannian metric on $L$. By Proposition~\ref{prop: good approximation}, a.e. $y \in G $ we have 
\begin{equation} \label{eq: goodpoint1}
    f_B(y) = \lim_{\epsilon \to 0^+} f^{d, \epsilon}_B(t)
\end{equation}
where $f^{d, \epsilon}_B$ is the $(d, \epsilon)$-density function of $B$. Using this observation and the inner regularity of the Haar measure on $L$, we obtain $\sigma$-compact $Y_\beta \subseteq \LL^+_{f_B}(\beta) $ such that $\lambda(Y_\beta) = \lambda(  \LL^+_{f_B}(\beta))$ and $Y_\beta$ consists only of $y \in G$ such that equation~\eqref{eq: goodpoint1} holds. The product $X_\alpha Y_\beta$ is $\sigma$-compact, and hence measurable. It remains to show that  
\begin{equation*} 
    \lambda( X_\alpha Y_\beta \setminus \LL^+_{f_{AB}}(\beta) ) =0
\end{equation*}
Again by Proposition~\ref{lem: good representation}, for a.e. $z\in X_\alpha Y_\beta$, we have
\begin{equation} \label{goodpoint2}
    f_{AB}(z) =  \lim_{\epsilon \to 0^+} f^{d, \epsilon}_{AB} (z).
\end{equation}
where $f^{d, \epsilon}_{AB}$ is the $(d, \epsilon)$-density function of $AB$ with respect to $\pi$.
It is enough to consider one such $z$ and show that $f_{AB}(z) \geq \beta$.  Let $x \in X_\alpha$ and $y\in Y_\beta$ be such that $z=xy$, and choose $a\in A$ such that $\pi(a) =x$. If $W\subseteq L$ is measurable, then 
$$\mu(aB \cap \pi^{-1}(W)) = \mu( B \cap \pi^{-1} (x^{-1} W)  ). $$
If $U_r(y)$ is a ball centered at $y$ with radius $r$, then by the left invariance of $d$, $xU(y)$ is the ball $U_r(z)$ centered at $z$ and with the same radius.
It follows that $f^{d, \epsilon}_B(x^{-1}t)$ is the $(d, \epsilon)$-density function for $aB$ with respect to $\pi$.

Since $aB \subseteq AB$, it implies that $f^{d, \epsilon}_B(x^{-1}t) \leq f^{d, \epsilon}_{AB}(t)$ for all $t \in L$.  In particular, with $t= z$, we get 
$$f^{d, \epsilon}_B(y) \leq f^{d, \epsilon}_{AB}(z)$$
sing equations~\eqref{eq: goodpoint1} and~\eqref{goodpoint2}, we get $f_B(y) \leq f_{AB}(z)$, so $f_{AB}(z) \geq \beta$ completing the proof.

Note that all the results in Section~8 have obvious counterparts with ``left'' replaced by ``right''. Moreover, the ``right'' counterparts follow from their ``left''-version by considering the opposite group $G^{\text{op}}$ where we define the group operation $\cdot^{\text{op}}$ by $g_1\cdot^{\text{op}}g_2:= g_2g_1$. Thus, we handle the case where $\alpha >\beta$ similarly switching the roles of $\LL^+_{f_A}(\alpha)$ and $\LL^+_{f_B}(\beta)$ and use a right-invariant metric instead of a left-invariant metric.
\end{proof}

We use the following simple  {\color{black}consequence of Fubini's theorem concerning} the superlevel sets~\cite[Theorem 8.16]{Rudin}:
\begin{fact}\label{fact: 4.1}
Let $\mu$ be a positive measure on some $\sigma$-algebra in the set $\Omega$. 
Suppose $f: \Omega\to [0,\infty]$ be a compactly support  measurable function. For every $r>0$, 
\[
\int_\Omega f^r(x)\d \mu(x)=\int_{\RR^{\geq0}} rx^{r-1} \mu(\LL_f^+(x))\d x.
\]
\end{fact}

We now show we can find subsets of the Lie model with small growth.

\begin{proposition} \label{Prop:smallgrowthinLiemodel}
Suppose $G$ is a pseudo-compact group equipped with a pseudo-Haar measure $\mu$, and $S\subseteq G$ is a open approximate group with connected Lie model $\pi: \langle S \rangle \to L$. Suppose $A, B \subseteq \langle S \rangle $ are $\mu$-measurable such that $\pi(A), \pi(B)$ are contained in compact sets. Then for all $\epsilon>0$, there are compact $X, Y \subseteq L$ such that $\BM(X,Y) \leq \BM(A,B) + \epsilon$.
\end{proposition}

\begin{proof}

If $L$ is compact we are done. Hence, we can assume that $L$ is noncompact. By Kemperman's inequality, for each $X, Y \subseteq L$, we have $\BM(X,Y) \geq 1$. 
Let $r=\BM(A,B)$ and assume there is $\varepsilon>0$ such that for every pair of sets $X,Y$ in $L$ we always have $\BM(X,Y)>r+\varepsilon$. 

Let $f_A$, $f_B$, and $f_{AB}$ denote the density function of $A$, $B$, and $AB$ respectively. 
Applying Lemma~\ref{lem: ReplacementBMinkernel}, for every $\alpha,\beta>0$ there are $X_\alpha\subseteq \LL^+_{f_A}(\alpha)\subseteq L$ and  $Y_\beta\subseteq \LL^+_{f_B}(\beta)\subseteq L$ such that $\lambda(X_\alpha)=\lambda( \LL^+_{f_A}(\alpha) )$, $\lambda(Y_\beta)=\lambda( \LL^+_{f_B}(\beta) )$, and
\begin{equation}\label{eq: density function growth tricially}
\lambda (X_\alpha Y_\beta\setminus \LL^+_{f_{AB}}(\max\{\alpha,\beta\}))=0. 
\end{equation}
By the assumption we have $\BM(X_\alpha,Y_\beta)>r+\varepsilon$. Together with \eqref{eq: density function growth tricially} we get
\begin{align}
 \lambda\left(\LL^{+}_{f_{AB}}\left(\max\{\alpha,\beta\}\right)\right)^{\frac{1}{r+\varepsilon}}&\geq  \lambda\left(X_\alpha Y_\beta \right)^{\frac{1}{r+\varepsilon}}\nonumber\\
 &\geq \lambda\left( X_\alpha \right)^{\frac{1}{r+\varepsilon}}+ \lambda\left(Y_\beta \right)^{\frac{1}{r+\varepsilon}}     = \lambda\left( \LL^+_{f_A}(\alpha) \right)^{\frac{1}{r+\varepsilon}}+ \lambda\left(\LL^+_{f_B}(\beta) \right)^{\frac{1}{r+\varepsilon}}. \label{eq: G when n_1=0}
\end{align}

To simplify the notation we define functions $F_A,F_B,F_{AB}:\RR\to\RR$ that 
\[
F_A(t)=\lambda(L^+_{f_A}(t))^{\frac{1}{r+\varepsilon}}, \quad F_B(t)=\lambda(L^+_{f_B}(t))^{\frac{1}{r+\varepsilon}}, \quad\text{and}\quad F_{AB}(t)=\lambda(L^+_{f_{AB}}(t))^{\frac{1}{r+\varepsilon}}.
\]
Clearly $F_A,F_B,F_{AB}$ are non-increasing functions, and hence measurable. By \eqref{eq: G when n_1=0}, we have 
\begin{equation*}
    F_{AB}(\max\{\alpha,\beta\})\geq F_A(\alpha)+F_B(\beta). 
\end{equation*}
 for all $\alpha,\beta \in \RR$. 
 This means, if we choose $\alpha,\beta\in\RR$ and assume that $F_A(\alpha)\geq t_1$ and $F_B(\beta)\geq t_2$, then $F_{AB}(\max\{\alpha,\beta\})\geq t_1+t_2$. Hence $\LL^+_{F_{A}}(t_1) \cup \LL^+_{F_{B}}(t_2) \subseteq \LL^+_{F_{AB}}(t_1+t_2)$. Let $\lambda_\RR$ be the Lebesgue measure on $\RR$, we thus have
\begin{equation}\label{relationofmuR}
\lambda_\RR(\LL^+_{F_{AB}}(t_1+t_2))\geq \max\{\lambda_\RR(\LL^+_{F_{A}}(t_1)),\lambda_\RR(\LL^+_{F_{B}}(t_2))\}.
\end{equation}

Finally, let us compute $\mu(AB)$. Using Fact~\ref{fact: 4.1} we have
\begin{align}\label{eq: the equation below 21}
    \mu(AB)^{\frac{1}{r+\varepsilon}}
    =\left(\int_{\RR^{>0}}F_{AB}^{r+\varepsilon}(t)\d t \right)^{\frac{1}{r+\varepsilon}}
    =\left(\int_{\RR^{>0}} (r+\varepsilon)t^{r+\varepsilon-1}\lambda_{\RR}(\LL_{F_{AB}}^+(t))\d t \right)^{\frac{1}{r+\varepsilon}}.
\end{align}    
Let $\Omega_A\subseteq \RR$ and $\Omega_B\subseteq \RR$ be the essential support of $F_A$ and $F_B$ respectively. Let $P_A=\lambda_\RR(\Omega_A) $ and $P_B=\lambda_\RR(\Omega_B)$. Using \eqref{relationofmuR} by letting $t_1$ and $t_2$ approach to $0$, the essential support of $F_{AB}$ has a $\lambda_\RR$-measure at least $P_A+P_B$. Thus by a change of variable we have
\begin{align*}
\left(\int_{\RR^{>0}} t^{r+\varepsilon-1}\lambda_{\RR}(\LL_{F_{AB}}^+(t))\d t\right)^{\frac{1}{r+\varepsilon}} &\geq \left(\int_{0}^{P_A+P_B} t^{r+\varepsilon-1}\lambda_{\RR}(\LL_{F_{AB}}^+(t))\d t\right)^{\frac{1}{r+\varepsilon}}\\
&= \left((P_A+P_B)^{r+\varepsilon}\int_{0}^{1} t^{r+\varepsilon-1}\lambda_{\RR}(\LL_{F_{AB}}^+((P_A+P_B)t))\d t\right)^{\frac{1}{r+\varepsilon}}.
\end{align*}
Now using \eqref{relationofmuR} again, the right hand side of the above inequality is at least
\[
\left((P_A+P_B)^{r+\varepsilon}\max\left\{\int_{0}^{1} t^{r+\varepsilon-1}\lambda_{\RR}(\LL_{F_{A}}^+(P_At))\d t, \int_{0}^{1} t^{r+\varepsilon-1}\lambda_{\RR}(\LL_{F_{B}}^+(P_Bt))\d t\right\}\right)^{\frac{1}{r+\varepsilon}}. 
\]
Using H\"older's inequality with exponents $(0,\infty)$, the above quantity is at least
\[
\left(P_A^{r+\varepsilon}\int_0^1 t^{r+\varepsilon-1}\lambda_{\RR}(\LL^+_{F_A}(P_A t))\d t \right)^{\frac{1}{r+\varepsilon}} + \left(P_B^{r+\varepsilon}\int_0^1 t^{r+\varepsilon-1}\lambda_{\RR}(\LL^+_{F_B}(P_B t))\d t\right)^{\frac{1}{r+\varepsilon}}.
\]
By a change of variable, and the definitions of $P_A$ and $P_B$, the above quantity is equal to
\[
\left(\int_{\RR^{>0}}t^{r+\varepsilon-1}\lambda_\RR(\LL_{F_A})(t)) \d t  \right)^{\frac{1}{r+\varepsilon}} + \left(\int_{\RR^{>0}}t^{r+\varepsilon-1}\lambda_\RR(\LL_{F_B})(t)) \d t  \right)^{\frac{1}{r+\varepsilon}}. 
\]
Therefore using \eqref{eq: the equation below 21} we obtain
\[
\mu(AB)^{\frac{1}{r+\varepsilon}}\geq \mu(A)^{\frac{1}{r+\varepsilon}}+\mu(B)^{\frac{1}{r+\varepsilon}}, 
\]
which contradicts the assumption that $\BM(A,B)=r$. 
\end{proof}

\section{Proof of the main theorems}

A main ingredient is the following measure growth gap result by the first two authors~\cite{JT}. 

\begin{fact}[Measure expansion gaps for semisimple Lie groups] \label{fact:expansiongap}
Suppose $G$ is a compact semisimple Lie group. 
If $A \subseteq G$ is compact with sufficiently small measure, then we have
$$ \mu(A^2) \geq (2+10^{-12}) \mu(A). $$
\end{fact}

We will also use the solution of the noncompact Kemperman Inverse Problem by An and the authors~\cite[Theorem 3.12]{AJTZ}. The main ingredient of this result is the nonabelian Brunn--Minkowski inequality in~\cite{JTZ} by the authors.

\begin{fact}[Kemperman inverse theorem for noncompact groups] \label{fact: AJTZ}
Let $G$ be a connected unimodular locally compact noncompact group equipped with a Haar measure $\mu$, and $A,B\subseteq G$ be a compact set satisfying $$\mu(AB)<\mu(A)+\mu(B)+2(\mu(A)\mu(B))^{\frac12}.$$ Then there is a continuous surjective group homomorphism $\chi: G\to\RR$ with compact kernel. 
\end{fact}

We prove a version of theorem~\ref{thm: mainBM} for simple Lie groups. The proof, in fact, also work for compact groups that does not have $\RR/\ZZ$ as a quotient by replacing Fact~\ref{fact:expansiongap} with its generalization from~\cite{JT}. As the exponent is not sharp for any group other than $\Sot$, we refrain from stating the result in the most general form for the sake of readability.

\begin{theorem}\label{thm: maingrowth2}
 For all $\epsilon>0$ and $N>0$, there is $c =c(\epsilon, N)$ such that if $G$ is a compact semisimple Lie group with normalized Haar measure $\mu$, and $A, B  \subseteq G$ are chosen with $A$, $B$, and $AB$ measurable, $0< \mu(A), \mu(B) < c$, and $\mu(A)/N< \mu(B)< N \mu(A)$, then
$$ \mu(AB)^{\frac{1}{2}+\epsilon}  > \mu(A)^{\frac{1}{2}+\epsilon}+ \mu(B)^{\frac{1}{2}+\epsilon}.$$
\end{theorem}

\begin{proof}
Fix $\epsilon$ and $N$ as in the statement of the theorem. Suppose to the contrary that no such $c$ exists. Then obtain a sequence $(G_n)$ of compact semisimple Lie groups, and sequence $(A_n)$ and $(B_n)$ of sets such that 
\begin{enumerate}
   \item[($1_n$)] $A_n, B_n, A_nB_n \subseteq G_n$ are measurable
   \item[($2_n$)] $\lim_{n \to \infty} \mu_n(A_n) =\lim_{n \to \infty} \mu_n(B_n) =0$, and $\mu_n(A_n)/N< \mu_n(B_n)< N\mu_n(A_n)$ where $\mu_n$ is the normalized Haar measure on $G_n$ 
    \item[($3_n$)] For each $n$,  $\mu(A_nB_n)^{\frac{1}{2}+\epsilon}  \leq  \mu(A_n)^{\frac{1}{2}+\epsilon}+ \mu(B_n)^{\frac{1}{2}+\epsilon}$, or equivalently,    
   $$\BM(A_n, B_n) \leq 2-4\epsilon',$$ 
    where we set $\epsilon'= \varepsilon/(1+2\varepsilon)$.
\end{enumerate}

Now choose an arbitrary nonprincipal ultrafilter $\ult$. Let $G$ be the ultraproduct $\prod_{\ult} G_n$ of the sequence $G_n$, and set $A=\prod_\ult A_n$ and $B =\prod_\ult B_n$. Applying Proposition~\ref{prop: smallexpansionultraproduct} we deduce that 
\begin{enumerate}
    \item $A$, $B$, and $AB$ are pseudo-measurable in $G$
    \item the sets $A$ and $B$ commensurable and infinitesimal compared to $G$.
    \item $\BM(A, B) \leq 2-4\epsilon'$.
\end{enumerate}
Noting also that $G$ is a pseudo-Lie pseudo-compact group, we apply Proposition~\ref{Prop: Approximategroupsfromsmallgrowth} to get $A', B', S \subseteq G$ such that the following conditions hold:

\begin{enumerate}
        \item[(1$'$)] Let $(G, \Sigma)$ be the expansion of the group $G$ generated by $A'$, $B'$, and $S$. Then every $D \subseteq G$ definable in $\Sigma$ is pseudo-measurable.  
        \item[(2$'$)] $A, B, A', B', S$ are all commensurable to one another
        \item[(3$'$)] $BM(A', B') \leq 2-3\epsilon'$
        \item[(4$'$)] $S$ is an approximate groups
        \item[(5$'$)]  $A'$ can be covered by finitely many left-translates of $S$, and $B'$ can be covered by finitely many right-translates of $S$.
    \end{enumerate}

Next, we apply Proposition~\ref{Prop: ChangingtoAPwithLiemodel} and replace $S$ if necessary to arrange that it has a connected Lie model $\pi: \langle S \rangle \to L$ continuously definable in $\Sigma$.

Recall that a semisimple Lie group is connected. Hence, $G$ is a pseudo-connected and pseudo-compact group. Conditions (1$'$) and (2$'$) then allow us to apply Proposition~\ref{Prop:SmallgrowthaboveLiemodel} to get $A'', B'' \subseteq \langle S \rangle$ such that the following hold:

\begin{enumerate}
        \item[(1$''$)]  $A''$ and  $B''$ are definable in $\Sigma$
        \item[(2$''$)] $A, B, A', B', A'', B'', S$ are all commensurable to one another
        \item[(3$''$)] $BM(A'', B'') \leq 2-2\epsilon'$.
    \end{enumerate}

Now apply Proposition~\ref{Prop:smallgrowthinLiemodel} to produce compact sets $X, Y \subseteq L$ such that 
$$\BM(X,Y) \leq 2-\epsilon'.$$
 Using Fact~\ref{fact: AJTZ}, there is a continuous and surjective group homomorphism with compact kernel $\phi: L \to \RR$.  
 
 Let $I \subseteq \RR$ be an open interval, we have $\lambda_\RR(I+I) =2 \lambda_\RR(I)$. Set $E = \pi^{-1} (\phi^{-1}( I ))$, we get $\mu(E^2) = 2\mu(E)$. Since $E$ is $\sigma$-definable in $\Sigma$, we obtain $D \subseteq E $ definable in $\Sigma$ such that 
 $$ \mu(D^2) < (2+10^{-12})\mu(D).  $$
 Since $D$ is definable in $\Sigma$, it is pseudo-measurable and is equal to $\prod_\ult D_n$ with $D_n \subseteq G_n$ measurable. Hence, there is $n$ sufficiently large such that $\mu_n(D^2_n) < (2 + 10^{-12}) \mu(D_n)$  and $\mu(D_n)$ arbitrarily small. This is a contradiction to Fact~\ref{fact:expansiongap}, which completes the proof.
\end{proof}

We now get a generalization of Theorem~\ref{thm: maingrowth}.

\begin{theorem}\label{thm: mainBM2}
 For every $\varepsilon>0$ there is $c>0$ such that whenever $G$ is a compact semisimple Lie group with normalized Haar measure $\mu$, $A \subseteq G$ is open with  $\mu(A)< c$, we have
$$ \mu(A^2)  > (4-\varepsilon) \mu(A).$$
\end{theorem}
\begin{proof}
Choose $\delta $ sufficiently small such that $ 2^{2-\delta}> 4-\varepsilon/10$, set $N=1$, and let $c =c(\epsilon, N)$ as in Theorem~\ref{thm: mainBM2}. Suppose $A \subseteq G$ is open, and $\mu(A)<c$. By the inner regularity of the Haar measure, one can choose compact $A' \subseteq A$ with $\mu(A') \geq (1- \varepsilon/10) \mu (A)$. Then we have
\[
\mu(A^2) \geq \mu((A')^2) \geq \left(4-\frac{\varepsilon}{10}\right) \mu(A') \geq \left(4-\frac{\varepsilon}{10}\right)\left(1-\frac{\varepsilon}{10}\right)\mu(A)>(4-\varepsilon)\mu(A). 
\]
This completes the proof.
\end{proof}

\begin{remark}\label{remark: section 10}
Fact~\ref{fact: AJTZ} has a natural generalization to all noncompact Lie groups if one can remove the helix dimension (defined in~\cite{JTZ}) term from the Brunn--Minkowski inequality; equivalently, if we have the following form of the Nonabelian Brunn--Minkowski Conjecture~\cite[Conjecture 1.4 and Theorem 1.5]{JTZ}. 
\begin{conjecture}[Nonabelian Brunn--Minkowski Conjecture]\label{conj: Nonabelian BM}
Let $G$ be a simply connected simple Lie group equipped with a Haar measure $\mu$, with $d$ the dimension of $G$ and $m$ the dimension of a maximal compact subgroup of $G$. Then for every compact sets $A,B\subseteq G$,
\[
\mu(AB)^{\frac{1}{d-m}}\geq \mu(A)^{\frac{1}{d-m}}+\mu(B)^{\frac{1}{d-m}}. 
\]
\end{conjecture}

The method developed in the paper is ready for us to prove the generalized Breuillard--Green conjecture for all compact simple Lie groups, if we have the Nonabelian Brunn--Minkowski Conjecture (Conjecture~\ref{conj: Nonabelian BM}) together with a suitable generalization of Fact~\ref{fact:expansiongap}, that $\mu(A^2)\geq (2^{d-m-1}+\eta)\mu(A)$ for some $\eta>0$.

It worth noting that although the Nonabelian Brunn-Minkowski Conjecture remains open, in \cite{JTZ} the authors proved the following theorem with an extra $2/3$ factor on the exponents: 
\begin{fact}[Nonabelian Brunn-Minkowski inequality for Lie groups] \label{nonabelBM}
Let $L$ be a connected Lie group with dimension $d$, let $m$ be the maximal dimension of a compact subgroup of $L$, and set $n=d-m$. Suppose, $L$ is unimodular and equipped with Haar measure $\mu$. For all $X,Y \subseteq L$ be compact with positive measure, we have
\[
 \mu(XY)^{ \lceil \frac{2n}{3} \rceil} \geq \mu(X)^{\lceil \frac{2n}{3} \rceil}  + \mu(Y)^{ \lceil \frac{2n}{3} \rceil}.  
 \]
\end{fact}

With Fact~\ref{nonabelBM} and a suitable generalization of Fact~\ref{fact:expansiongap}, it is possible to use our method  to show a weaker version of Conjecture~\ref{conj: generalization} that $\mu(A^2)>(2^{\frac{2(n-m)}{3}}-\varepsilon)\mu(A)$ for sufficiently small $A$. 
\end{remark}

\section*{Acknowledgements}
The authors thank Arturo Rodriguez Fanlo, Ben Green, Ehud Hrushovski, Simon Machado, and Jinhe Ye for discussions.

\bibliographystyle{amsalpha}
\bibliography{ref}

\end{document}